\documentclass{article}
\usepackage[latin1]{inputenc}
\usepackage[frenchb,english]{babel}
\usepackage[T1]{fontenc}
\usepackage{amsmath}
\usepackage{titlesec}
\numberwithin{equation}{section}
\usepackage{geometry}
\usepackage{amssymb}
\usepackage{pifont}
\usepackage{dsfont}
\usepackage{verbatim}
\usepackage{graphicx}
\usepackage{color}
\usepackage{mathrsfs}
\usepackage[citecolor=blue,colorlinks=true]{hyperref}

\newcommand{\R}{\mathbb{R}}

\newcommand{\N}{\mathbb{N}}
\newcommand{\PP}{\mathbb{P}}
\newcommand{\E}{\mathbb{E}}
\newcommand{\eps}{\varepsilon}
\newcommand{\LL}{L^2_{F^{-1}}}
\newcommand{\dd}{\textrm{d}}

\newtheorem{theorem}{Theorem}[section]
\newtheorem{lemma}[theorem]{Lemma}
\newtheorem{prop}[theorem]{Proposition}

\newtheorem{definition}{Definition}[section]

\newenvironment{proof}[1][Proof]{\begin{trivlist}
\item[\hskip \labelsep {\bfseries #1}]}{\end{trivlist}}

\begin{document}
\begin{center}
\textbf{FRACTIONAL DIFFUSION LIMIT FOR A STOCHASTIC KINETIC EQUATION}
\end{center}

\vspace{0.2cm}
\begin{center}
\small \sc{Sylvain De Moor}
\end{center}

\begin{center}
\footnotesize{ENS Cachan Bretagne - IRMAR, Université Rennes 1,\\ Avenue Robert Schuman, F-35170 Bruz, France}\\
Email: \texttt{sylvain.demoor@ens-rennes.fr}
\end{center}
\vspace{0.2cm}

\begin{abstract}
\footnotesize We study the stochastic fractional diffusive limit of a kinetic equation involving a small parameter and perturbed by a smooth random term. Generalizing the method of perturbed test functions, under an appropriate scaling for the small parameter, and with the moment method used in the deterministic case, we show the convergence in law to a stochastic fluid limit involving a fractional Laplacian.\\

\end{abstract}

\normalsize

\section{Introduction}
In this paper, we consider the following equation 
\begin{equation}
\label{eq}
\partial_tf^{\eps}+\frac{1}{\eps^{\alpha-1}}v\cdot\nabla_xf^{\eps}=
\frac{1}{\eps^{\alpha}}Lf^{\eps}+\frac{1}{\eps^{\frac{\alpha}{2}}}m^{\eps}f^{\eps} \;\;\mathrm{in}\;\R^+_t\times\R^d_x\times \R^d_v,
\end{equation}

\noindent with initial condition 
\begin{equation}\label{eq0}
f^{\eps}(0)=f^{\eps}_0\;\;\mathrm{in}\;\;\R^d_x\times \R^d_v,
\end{equation}

\noindent where $0<\alpha<2$, $L$ is a linear operator (see $(\ref{defdeL})$ below) and $m^{\eps}$ a random process depending on $(t,x)\in\R^+\times \R^d$ (see Section \ref{sectionmeps}). We will study the behaviour in the limit $\eps\to 0$ of its solution $f^{\eps}$. \\

\noindent The solution $f^{\eps}(t,x,v)$ to this kinetic equation can be interpreted as a distribution function of particles having position $x$ and degrees of freedom $v$ at time $t$. The variable $v$ belongs to the velocity space $\R^d$ that we denote by $V$.\\
The collision operator $L$ models diffusive and mass-preserving interactions of the particles with the surrounding
medium; it is given by 
\begin{equation}\label{defdeL}
Lf=\int_Vf\,\dd v\, F-f,
\end{equation}
where $F$ is a velocity equilibrium function such that $F\in L^{\infty}$, $F(-v)=F(v)$, $F>0$ a.e., $\int_V F(v)\dd v = 1$ and which is a power tail distribution
\begin{equation}\label{powertail}
F(v)\underset{|v|\to\infty}{\sim}\frac{\kappa_0}{|v|^{d+\alpha}}.
\end{equation}
Note that $F\in\ker(L)$. Power tail distribution functions arise in various contexts, such as astrophysical plasmas or in the study of granular media. For more details on the subject, we refer to \cite{MMM}. \\

\noindent In this paper, we derive a stochastic diffusive limit to the random kinetic model $(\ref{eq})-(\ref{eq0})$, using the method of perturbed test functions. This method provides an elegant way of deriving stochastic diffusive limit from random kinetic systems; it was first introduced by Papanicolaou, Stroock and Varadhan \cite{psv}. The book of Fouque, Garnier, Papanicolaou and Solna \cite{fgps} presents many applications to this method. A generalisation in infinite dimension of the perturbed test functions method arose in recent papers of Debussche and Vovelle \cite{arnaudjulien} and de Bouard and Gazeau \cite{debouard}. \\
For the random kinetic model $(\ref{eq})-(\ref{eq0})$, the case $\alpha=2$ and $v$ replaced by $a(v)$ where $a$ is bounded is derived in the paper of Debussche and Vovelle  \cite{arnaudjulien}. Here we study a different scaling parametrized by $0<\alpha<2$ and we relax the boundedness hypothesis on $a$ since we study the case $a(v)=v$. Note that, in our case, in order to get a non-trivial limiting equation as $\eps$ goes to $0$, we exactly must have $a(v)$ unbounded; furthermore, we can easily extend the result to velocities of the form $a(v)$ where $a$ is a $C^1$-diffeomorphism from $V$ onto $V$. In the deterministic case, i.e. $m^{\eps}\equiv 0$, Mellet derived in \cite{mellet} and \cite{MMM} with Mouhot and Mischler a diffusion limit to this kinetic equation involving a fractional Laplacian. As a consequence, for the random kinetic problem $(\ref{eq})-(\ref{eq0})$, we expect a limiting stochastic equation with a fractional Laplacian. \\
As in the deterministic case, the fact that the equilibrium $F$ have an appropriate growth when $|v|$ goes to $+\infty$, namely of order $|v|^{-d-\alpha}$, is essential to derive a non-trivial limiting equation when $\eps$ goes to $0$. \\
To derive a stochastic diffusive limit to the random kinetic model $(\ref{eq})-(\ref{eq0})$, we use a generalisation in infinite dimension of the perturbed test functions method. Nevertheless, the fact that the velocities are not bounded gives rise to non-trivial difficulties to control the transport term $v\cdot\nabla_x$. As a result, we also use the moment method applied in \cite{mellet} in the deterministic case. The moment method consists in working on weak formulations and in introducing new auxiliary problems, namely in the deterministic case 
$$
\chi^{\eps}-\eps v\cdot\nabla_x\chi^{\eps}=\varphi,
$$
where $\varphi$ is some smooth function;
 thus we introduce in the sequel several additional auxiliary problems to deal with the stochastic part of the kinetic equation. Solving these problems is based on the inversion of the operator $L-\eps A+M$ where $M$ is the infinitesimal generator of the driving process $m$. Finally, we have to combine appropriately the moment and the perturbed test functions methods.\\

\noindent We also point out similar works using a more probabilistic approach of Basile and Bovier \cite{basilebovier} and Jara, Komorowski and Olla \cite{jara}.

\section{Preliminaries and main result}
\subsection{Notations}\label{notations}
In the sequel, $\LL$ denotes the $F^{-1}$ weighted $L^2(\R^d\times V)$ space equipped with the norm  
$$\|f\|^2:=\int_{\R^d}\!\int_{V}\frac{|f(x,v)|^2}{F(v)}\,\dd v\dd x.$$
We denote its scalar product by $(.,.)$. We also need to work in the space $L^2(\R^d)$, or $L^2_x$ for short. The scalar product in $L^2_x$ will be denoted by $(.,.)_x$. When $f\in\LL$, we denote by $\rho$ the first moment of $f$ over $V$ i.e. $\rho=\int_Vf\,\dd v$. We often use the following inequality 
$$\|\rho\|_{L^2_x}\leq\|f\|,$$
which is just Cauchy-Schwarz inequality and the fact that $\int_V F(v)\,\dd v=1$. 
\noindent Finally, $\mathcal{S}(\R^d)$ stands for the Schwartz space on $\R^d$, and $\mathcal{S}'(\R^d)$ for the space of tempered distributions on $\R^d$. \\

\noindent We recall that the operator $L$ is defined by $(\ref{defdeL})$. It can easily be seen that $L$ is a bounded operator from $\LL$ to $\LL$. Note also that $L$ is dissipative since, for $f\in\LL$, 
\begin{equation}\label{dissip}
(Lf,f)=-\|Lf\|^2.
\end{equation}
In the sequel, we denote by $g(t,\cdot)$ the semi-group generated by the operator $L$ on $\LL$. It satisfies, for $f\in\LL$,
$$
\left\{
\begin{aligned}
\frac{d}{\dd t}g(t,f)&=Lg(t,f), \\
g(0,f)&=f,
\end{aligned}
\right.
$$
and it is given by
$$g(t,f)=\int_V f\,\dd v \,F(1-e^{-t})+fe^{-t},\quad t\geq 0,\; f\in\LL,$$
so that $g(t,\cdot)$ is a contraction, that is, for $f\in\LL$,
\begin{equation}\label{contract}
\|g(t,f)\|\leq\|f\|,\quad t\geq 0.
\end{equation}

\noindent We now introduce the following spaces $S^{\gamma}$ for $\gamma\in\R$.
First, we define the following operator on $L^2(\R^d)$
$$J:=-\Delta_x+|x|^2,$$
with domain $$\mathrm{D}(J):=\left\lbrace f\in L^2(\R^d),\;\Delta_xf,\; |x|^2f\in L^2(\R^d)\right\rbrace.$$ 
Let $(p_j)_{j\in\N^d}$ be the Hermite functions, defined as
$$p_j(x_1,...,x_d):=H_{j_1}(x_1)\cdots H_{j_d}(x_d)e^{-\frac{|x|^2}{2}},$$
where $j=(j_1,...,j_d)\in\N^d$ and $H_i$ stands for the $i-$th Hermite's polynomial on $\R$. The functions $(p_j)_{j\in\N^d}$ are the eigenvectors of $J$ with associated eigenvalues $(\mu_j)_{j\in\N^d}:=(2|j|+1)_{j\in\N^d}$ where $|j|:=|j_1|+\cdots+|j_d|$. Furthermore, one can check that $J$ is invertible from $\mathrm{D}(J)$ to $L^2(\R^d)$, and that it is self-adjoint. As a result, we can define $J^{\gamma}$ for any $\gamma\in\R$.\\
Then, for $\gamma\in\R$, we can also view $J^{\gamma}$ as an operator on $\mathcal{S}'(\R^d)$. Let $u\in\mathcal{S}'(\R^d)$, we define $J^{\gamma}u\in\mathcal{S}'(\R^d)$ by setting, for all $\varphi\in\mathcal{S}(\R^d)$,
$$\langle J^{\gamma}u,\varphi\rangle:=\langle u,J^{\gamma}\varphi\rangle.$$
Finally, we introduce, for $\gamma\in\R$,
$$S^{\gamma}(\R^d):=\{u\in\mathcal{S}'(\R^d),\;J^{\frac{\gamma}{2}}u\in L^2(\R^d)\},$$
equipped with the norm $$\|u\|_{S^{\gamma}(\R^d)}=\|J^{\frac{\gamma}{2}}u\|_{L^2(\R^d)}.$$
In the sequel, we need to know the asymptotic behaviour of the quantities $\|p_j\|_{L^2_x}$, $\|\nabla_xp_j\|_{L^2_x}$, $\|D^2p_j\|_{L^2_x}$ and $\|(-\Delta)^{\frac{\alpha}{2}}p_j\|_{L^2_x}$ as $|j|\to\infty$. In fact, classical properties of the Hermite functions give the following bounds
\begin{equation}\label{taillej}
\begin{aligned}
&\|p_j\|_{L^2_x}=1, \quad && \|\nabla_xp_j\|_{L^2_x}\leq {\mu_j}^{\frac{1}{2}},\\
&\|D^2p_j\|_{L^2_x}\leq \mu_j, \quad && \|(-\Delta)^{\frac{\alpha}{2}}p_j\|_{L^2_x}\leq 1+\mu_j.
\end{aligned}
\end{equation} 
We finally recall the definition of the fractional power of the Laplacian. It can be introduced using the Fourier transform in $\mathcal{S}'(\R^d)$ by setting, for $u\in\mathcal{S}'(\R^d)$,
$$\mathcal{F}((-\Delta)^{\frac{\alpha}{2}}u)(\xi)=|\xi|^{\alpha}\mathcal{F}(u)(\xi).$$
Alternatively, we have the following singular integral representation, see \cite{laplfrac},
$$-(-\Delta)^{\frac{\alpha}{2}}u(x)=c_{d,\alpha}\mathrm{PV}\int_{\R^d}\left[u(x+h)-u(x)\right]\frac{\dd h}{|h|^{d+\alpha}},$$
for some constant $c_{d,\alpha}$ which only depends on $d$ and $\alpha$.

\subsection{The random perturbation}\label{sectionmeps}

The random term $m^{\eps}$ is defined by 
$$m^{\eps}(t,x):=m\left(\frac{t}{\eps^{\alpha}},x\right),$$ 
where $m$ is a stationary process on a probability space $(\Omega,\mathcal{F},\PP)$ and is adapted to a filtration $(\mathcal{F}_t)_{t\geq 0}$. Note that $m^{\eps}$ is adapted to the filtration $(\mathcal{F}^{\eps}_t)_{t\geq 0}=(\mathcal{F}_{\eps^{-\alpha}t})_{t\geq 0}.$ We assume that, considered as a random process with values in a space of spatially dependent functions, $m$ is a stationary homogeneous Markov process taking values in a subset $E$ of $L^2(\R^d)\cap W^{1,\infty}(\R^d)$. In the sequel, $E$ will be endowed with the norm $\|\cdot\|_{\infty}$ of $L^{\infty}(\R^d)$. Besides, we denote by $\mathcal{B}(E,X)$ (or $\mathcal{B}(E)$ when $X=\R$) the set of bounded functions from $E$ to $X$ endowed with the norm $\|g\|_{\infty}:=\sup_{n\in E}\|g(n)\|_X$ for $g\in\mathcal{B}(E,X)$. \\
We assume that $m$ is stochastically continuous. Note that $m$ is supposed not to depend on the variable $v$. For all $t\geq 0$, the law $\nu$ of $m_t$ is supposed to be centered
$$\E[m_t]=\int_E n\,\dd \nu(n)=0.$$
The subset $E$ has the following properties. We fix a family $(\eta_i)_{i\in\N}$ of functions in $W^{1,\infty}(\R^d)$ such that 
$$S:=\sum\limits_{i\in\N}\|\eta_i\|_{W^{1,\infty}}<\infty,$$
and we assume that every $n\in E$ can be uniquely written as
\begin{equation}\label{decompositionden}
n=\sum\limits_{i\in\N}n_i(n)\eta_i,
\end{equation}
with $|n_i(n)|\leq K$ for all $i\in\N$ and all $n\in E$. Note that the preceding series converges absolutely and that $E$ is included in the ball $\mathrm{B}(0,KS)$ of $W^{1,\infty}(\R^d)$. Finally, since $m$ is centered, we also suppose that for all $i\in\N$, \begin{equation}\label{nicentres}
\int_En_i(n)\dd\nu(n)=0.
\end{equation}

\noindent We denote by $e^{tM}$ a transition semi-group on $E$ associated to $m$. We suppose that the transition semi-group is Feller i.e. $e^{tM}$ maps continuous functions of $n$ on continuous functions of $n$ for all $t\geq 0$. In the sequel we also need to consider $e^{tM}$ as a transition semi-group on the space $\mathcal{B}(E,\LL)$ and not only on $\mathcal{B}(E)$. Thus, if $g\in \mathcal{B}(E,\LL)$, $e^{tM}$ acts on $g$ pointwise, that is, 
$$[\widetilde{e^{tM}}g](x,v)=e^{tM}[g(x,v)],\quad (x,v)\in\R^d\times V.$$ 
In both cases, we denote by $M$ the infinitesimal generator associated to the transition semi-group. Note that we do not distinguish on which space $\mathcal{B}(E,X)$, $X=\R$ or $\LL$, the operators are acting since it will always be clear from the context. Then, for $X=\R$ or $X=\LL$, $\mathrm{D}(M)$ stands for the domain of $M$; it is defined as follows:
$$\mathrm{D}(M):=\left\{u\in \mathcal{B}(E,X), \,\lim\limits_{h\to 0}\frac{e^{hM}-I}{h}u \textrm{ exists in } \mathcal{B}(E,X)\right\},$$
and if $u\in\mathrm{D}(M)$, we set
$$Mu:=\lim\limits_{h\to 0}\frac{e^{hM}-I}{h}u \textrm{ in } \mathcal{B}(E,X).$$

\noindent We suppose that there exists $\mu>0$ such that for all $g\in \mathcal{B}(E)$ verifying the condition $\int_Eg(n)\dd\nu(n)=0$,
\begin{equation}\label{decroiss}
\|e^{tM}g\|_{\infty}\leq e^{-\mu t}\|g\|_{\infty},\quad t\geq 0.
\end{equation}

\noindent Moreover, we suppose that $m$ is ergodic and satisfies some mixing properties in the sense that there exists a subspace $\mathscr{P}_M$ of $\mathcal{B}(E)$ such that for any $g\in\mathscr{P}_M$, the Poisson equation
$$M\psi=g-\int_Eg(n)\,\dd\nu(n)=:\widehat{g},$$
has a unique solution $\psi\in\mathrm{D}(M)$ satisfying $\int_E\psi(n)\,\dd\nu(n)=0$. We denote by $M^{-1}\widehat{g}$ 
this unique solution, and assume that it is given by
\begin{equation}\label{inversedeM}
M^{-1}\widehat{g}(n)= \int_0^{\infty}e^{tM}\widehat{g}(n)\dd t,\quad n\in E.
\end{equation}
\noindent Thanks to $(\ref{decroiss})$, the above integral is well defined. In particular, it implies that for all $n\in E$,
$$\lim\limits_{t\to\infty}e^{tM}\widehat{g}(n)= 0.$$
We assume that for all $i\in\N$, $n\mapsto n_i(n)$ is in $\mathscr{P}_M$ and that for all $n\in E$, $|M^{-1}n_i(n)|\leq K$. As a consequence, we simply define $M^{-1}I$ by
$$M^{-1}I(n):=\sum_{i\in\N}M^{-1}n_i(n)\eta_i,\quad n\in E.$$

\noindent We also suppose that for all $f\in L^2(\R^d)$, the functions $g_f:n\in E\mapsto(f,n)_x$ and $n\in E\mapsto M^{-1}g_f(n)$ are  in $\mathscr{P}_M$.\\

\noindent We will suppose that for all $t\geq 0$,
\begin{equation}\label{pournoyau}
\E\|m_t\|^2_{L^2_x}<\infty,\quad \E\|M^{-1}I(m_t)\|^2_{L^2_x}<\infty.
\end{equation}
To describe the limiting stochastic PDE, we then set 
$$k(x,y)=\E\int_{\R}m_0(y)m_t(x)\,\dd t,\quad x,y\in\R^d.$$
The kernel $k$ is, thanks to $(\ref{pournoyau})$, the fact that $m$ is stationary and Cauchy-Schwarz inequality, in $L^2(\R^d\times \R^d)$ and such that 
$$\int_{\R^d} k(x,x)\,\dd x<\infty.$$ Furthermore, we can check (see \cite{arnaudjulien}), since $m$ is stationary, that $k$ is symmetric. As a result, we introduce the operator $Q$ on $L^2(\R^d)$ associated to the kernel $k$ 
$$Qf(x)=\int_{\R^d}k(x,y)f(y)\,\dd y,\quad x\in\R^d,$$
which is self-adjoint and trace class. Furthermore, since we assumed that the functions $g_f:n\in E\mapsto(f,n)_x$ and $n\in E\mapsto M^{-1}g_f(n)$ are  in $\mathscr{P}_M$, we can show, see \cite[Lemma 1]{arnaudjulien}, that $Q$ is non-negative, that is $(Qf,f)_x\geq 0$ for all $f\in L^2(\R^d)$. As a result, we can define the square root $Q^{\frac{1}{2}}$ which is Hilbert-Schmidt on $L^2(\R^d)$.\\

\noindent It remains to make some hypothesis on $M$. We set, for all $n\in E$,
\begin{equation}\label{deftheta}
\theta(n)=\int_EnM^{-1}I(n)\dd\nu(n)-nM^{-1}I(n),
\end{equation}
and, for $i,j\in\N$, $\theta_{i,j}=\int_En_iM^{-1}n_j\dd\nu-n_iM^{-1}n_j$, so that 
$$\theta=\sum_{i,j\in\N}\theta_{i,j}\eta_i\eta_j.$$

\noindent We suppose that for all $i,j,k,l\in\N$ and $s,t \geq 0$, 
\begin{equation}\label{dansdomaine}
\left\{
\begin{aligned}
&n\mapsto n_i(n),\\
&n\mapsto \theta_i(n),\\
&n\mapsto e^{tM}n_i(n)e^{sM}n_j(n),\\
&n\mapsto e^{tM}\theta_{i,j}(n)e^{sM}n_k(n),\\
&n\mapsto e^{tM}\theta_{i,j}(n)e^{sM}\theta_{k,l}(n),
\end{aligned}
\right.
\end{equation}
are in $\mathrm{D}(M)$,
with 
\begin{equation}\label{borneM1}
\|n_i\|_{\infty}+\|\theta_{i,j}\|_{\infty}+\|Mn_i\|_{\infty}+\|M\theta_{i,j}\|_{\infty}\leq K,
\end{equation}
\begin{equation}\label{borneM2}
\|M[e^{tM}n_ie^{sM}n_j]\|_{\infty}+\|M[e^{tM}\theta_{i,j}e^{sM}n_j]\|_{\infty}+\|M[e^{tM}\theta_{i,j}e^{sM}\theta_{k,l}]\|_{\infty}\leq Ke^{-\mu(s+t)}.
\end{equation}

\noindent \textbf{Remark} The above assumptions $(\ref{decroiss})-(\ref{borneM2})$ on the process $m$ are verified, for instance, when $m$ is a Poisson process taking values in $E$. \\

\noindent We now state two lemmas which will be very useful in the following.
\begin{lemma}\label{ehmmoinsiborne}
Let $p\in \mathcal{B}(E)$ be a function in $\mathrm{D}(M)$ such that $\|Mp\|_{\infty}\leq K$. Then we have, for all $h>0$,
$$\left\|\frac{e^{hM}-I}{h}p-Mp\right\|_{\infty}\leq 2K.$$
\end{lemma}
\begin{proof}
We just write, for all $n\in E$,
\begin{alignat*}{2}
\left|\frac{e^{hM}-I}{h}p(n)-Mp(n)\right|&=\left|\frac{1}{h}\int_0^hMe^{sM}p(n)\,\dd s-Mp(n)\right| \\
&=\left|\frac{1}{h}\int_0^he^{sM}Mp(n)\,\dd s-Mp(n)\right|\leq 2K,
\end{alignat*}
where we used the contraction property of the semigroup $e^{tM}$. This concludes the proof. $\Box$
\end{proof}

\noindent \textbf{Remark} The proof is still valid if $p\in \mathcal{B}(E,\LL)$; we just have to replace the absolute values by the $\LL$-norm.

\begin{lemma}\label{continuiteenn}
For all $i,j,k,l\in\N$ and $s,t \geq 0$, the functions
\begin{equation}
\begin{aligned}
\left\{
\begin{aligned}
&n\mapsto n_i(n), \\
&n\mapsto \theta_{i,j}(n),  \\
&n\mapsto e^{tM}n_ie^{sM}n_j(n), \\
&n\mapsto e^{tM}\theta_{i,j}e^{sM}n_k(n),\\
&n\mapsto e^{tM}\theta_{i,j}e^{sM}\theta_{k,l}(n),
\end{aligned}
\right.&\quad \left\{
\begin{aligned}
&n\mapsto Mn_i(n), \\
&n\mapsto M\theta_{i,j}(n), \\
&n\mapsto M[e^{tM}n_ie^{sM}n_j](n),\\
&n\mapsto M[e^{tM}\theta_{i,j}e^{sM}n_k](n),\\
&n\mapsto M[e^{tM}\theta_{i,j}e^{sM}\theta_{k,l}](n),
\end{aligned}
\right.
\end{aligned}
\end{equation}
are continuous.
\end{lemma}

\begin{proof}
We fix $i,j,k,l\in\N$ and $s,t \geq 0$. First of all, $n\mapsto n_i(n)$ is obviously continuous since it is linear. We recall that $\theta_{i,j}=\int_En_iM^{-1}n_j\dd\nu-n_iM^{-1}n_j$. With $(\ref{nicentres})$ and $(\ref{inversedeM})$, we have 
$$M^{-1}n_j=\int_0^{\infty}e^{tM}n_j\,\dd t,$$
which is continuous with respect to $n\in E$ by $(\ref{nicentres})$, $(\ref{decroiss})$, $(\ref{borneM1})$ and the dominated convergence theorem. As a result, $n\mapsto n_i(n)M^{-1}n_j(n)$ is continuous; and also $n\mapsto \int_En_i(n)M^{-1}n_j(n)\dd\nu(n)$ by the dominated convergence theorem. Hence $n\mapsto\theta_{i,j}(n)$ is continuous. The continuity of $n_i$ and $\theta_{i,j}$ now immediately gives the continuity of the three last functions of the left group of the lemma by the Feller property of the semigroup $e^{tM}$. \\
For the remaining functions, just remark that if $p\in \mathcal{B}(E)$ is in $\mathrm{D}(M)$ and continuous, then $Mp$ is the uniform limit on $E$ when $h\to 0$ of the functions 
$$\frac{e^{hM}-I}{h}p,$$
which are continuous by the Feller property of the semigroup. Hence $Mp$ is continuous. This ends the proof. $\Box$
\end{proof}

\subsection{Resolution of the kinetic equation}
In this section, we solve the linear evolution problem $(\ref{eq})-(\ref{eq0})$ thanks to a semigroup approach. We thus introduce the linear operator $A:=-v\cdot\nabla_x$ on $L^2_{F^{-1}}$ with domain $$\mathrm{D}(A):=\{f\in L^2_{F^{-1}},\;v\cdot\nabla_xf\in L^2_{F^{-1}}\}.$$
The operator $A$ has dense domain and, since it is skew-adjoint, it is $m$-dissipative. Consequently $A$ generates a contraction semigroup $(\mathcal{T}(t))_{t\geq 0}$, see \cite{cazhar}. We recall that $\mathrm{D}(A)$ is endowed with the norm $\|\cdot\|_{\mathrm{D}(A)}:=\|\cdot\|+\|A\cdot\|$, and that it is a Banach space.

\begin{prop}\label{solkinetic}Let $T>0$ and $f_0^{\eps}\in L^2_{F^{-1}}$. Then there exists a unique mild solution of $(\ref{eq})-(\ref{eq0})$ on $[0,T]$ in $L^{\infty}(\Omega)$, that is there exists a unique $f^{\eps}\in L^{\infty}(\Omega,\mathcal {C}([0,T],\LL))$ such that $\PP-$a.s. $$f^{\eps}_t=\mathcal{T}\left(\frac{t}{\eps^{\alpha-1}}\right)f_0^{\eps}+\int_0^t\mathcal{T}\left(\frac{t-s}{\eps^{\alpha-1}}\right)\left(\frac{1}{\eps^{\alpha}}Lf^{\eps}_s+\frac{1}{\eps^{\frac{\alpha}{2}}}m^{\eps}_sf^{\eps}_s\right)\,ds,\quad t\in[0,T].$$
Assume further that $f^{\eps}_0\in \mathrm{D}(A)$, then there exists a unique strong solution $f^{\eps}\in L^{\infty}(\Omega, C^1([0,T],\LL))\cap L^{\infty}(\Omega, C([0,T],\mathrm{D}(A)))$ of $(\ref{eq})-(\ref{eq0})$.
\end{prop}
\begin{proof}
Subsections 4.3.1 and 4.3.3 in \cite{cazhar} gives that $\PP-$a.s. there exists a unique mild solution $f^{\eps}\in \mathcal {C}([0,T],\LL)$ and it is not difficult to slightly modify the proof to obtain that in fact $f^{\eps}\in L^{\infty}(\Omega,\mathcal {C}([0,T],\LL))$ (we intensively use that for all $t\geq 0$ and $\eps>0$, $\|m^{\eps}_t\|_{\infty}\leq K$). \\
Similarly, subsections 4.3.1 and 4.3.3 in \cite{cazhar} gives us $\PP-$a.s. a strong solution $f^{\eps}\in  C^1([0,T],\LL)\cap  C([0,T],\mathrm{D}(A))$ of $(\ref{eq})-(\ref{eq0})$ and once again one can easily get that in fact $f^{\eps}\in L^{\infty}(\Omega, C^1([0,T],\LL))\cap L^{\infty}(\Omega, C([0,T],\mathrm{D}(A)))$. $\Box$
\end{proof}


\subsection{Main result}
We are now ready to state our main result:
\begin{theorem}\label{mainresult}
Assume that $(f^{\eps}_0)_{\eps>0}$ is bounded in $\LL$ and that 
$$\rho^{\eps}_0:=\int_Vf^{\eps}_0\,\dd v\underset{\eps\to 0}{\longrightarrow} \rho_0 \text{ in }L^2(\R^d).$$
Then, for all $\eta>0$ and $T>0$, $\rho^{\eps}:=\int_Vf^{\eps}\,\dd v$ converges in law in $ C([0,T],S^{-\eta})$ to the solution $\rho$ to the stochastic diffusion equation
\begin{equation}\label{stochasticeq}
d\rho=-\kappa(-\Delta)^{\frac{\alpha}{2}}\rho \dd t+\frac{1}{2}H\rho +\rho Q^{\frac{1}{2}}dW_t, \text{ in } \R^+_t\times \R^d_x,
\end{equation}
with initial condition $\rho(0)=\rho_0$ in $L^2(\R^d)$, and where $W$ is a cylindrical Wiener process on $L^2(\R^d)$, \begin{equation}
\label{defkappa}
\kappa:=\frac{\kappa_0}{c_{d,\alpha}}\int_0^{\infty}|t|^{\alpha}e^{-t}\,\dd t,
\end{equation}
and 
\begin{equation}
\label{defH}
H:=\int_E nM^{-1}I(n)\,\dd\nu(n)\in W^{1,\infty}.
\end{equation}
\end{theorem}
\textbf{Remark} The limiting equation $(\ref{stochasticeq})$ can also be written in Stratonovitch form
$$
d\rho=-\kappa(-\Delta)^{\frac{\alpha}{2}}\rho \dd t+\rho\circ Q^{\frac{1}{2}}dW_t.
$$

\noindent \textbf{Notation} In the sequel, we will note $\lesssim$ the inequalities which are valid up to constants of the problem, namely $K$, $S$, $\mu$, $d$, $\alpha$, $\|L\|$, $\sup_{\eps>0}\|f^{\eps}_0\|$ and real constants. Nevertheless, when we need to emphasize the dependence of a constant on a parameter, we index the constant $C$ by the parameter; for instance the constant $C_{\varphi}$ depends on $\varphi$.

\section{The generator}
The process $f^{\eps}$ is not Markov (indeed, by $(\ref{eq})$, we need $m^{\eps}$ to know the increments of $f^{\eps}$) but the couple $(f^{\eps},m^{\eps})$ is. From now on, we denote by $\mathscr{L}^{\eps}$ its infinitesimal generator, it is given by
$$\mathscr{L}^{\eps}\Psi(f,n)=\frac{1}{\eps^{\alpha}}(Lf+\eps Af,D\Psi(f,n))+\frac{1}{\eps^{\frac{\alpha}{2}}}(fn,D\Psi(f,n))+\frac{1}{\eps^{\alpha}}M\Psi(f,n),$$
provided $\Psi:\LL\times E\to\R$ is enough regular to be in the domain of $\mathscr{L}^{\eps}$. Thus we begin this section by introducing a special set of functions which lie in the domain of  $\mathscr{L}^{\eps}$ and satisfy the associated martingale problem. In the following, if $\Psi:\LL\rightarrow\R$ is differentiable with respect to $f\in\LL$, we denote by $D\Psi(f)$ its differential at a point $f$ and we identify the differential with the gradient. 
\begin{definition}\label{goodtest}We say that $\Psi:\LL\times E\rightarrow\R$ is a good test function if 
\begin{enumerate}
\item[$(i)$]$(f,n)\mapsto\Psi(f,n)$ is differentiable with respect to $f$;
\item[$(ii)$]$(f,n)\mapsto D\Psi(f,n)$ is continuous from $\LL\times E$ to $\LL$ and maps bounded sets onto bounded sets;
\item[$(iii)$]for any $f\in\LL$, $\Psi(f,\cdot)\in D_M$;
\item[$(iv)$]$(f,n)\mapsto M\Psi(f,n)$ is continuous from $\LL\times E$ to $\R$ and maps bounded sets onto bounded sets.
\end{enumerate}
\end{definition}

\begin{prop}\label{gene}Let $\Psi$ be a good test function. 
If $f^{\eps}_0\in\mathrm{D}(A)$,
$$M^{\eps}_{\Psi}(t):=\Psi(f^{\eps}_t,m^{\eps}_t)-\Psi(f^{\eps}_0,m^{\eps}_0)-\int_0^t\mathscr{L}^{\eps}\Psi(f^{\eps}_s,m^{\eps}_s)\,\dd s$$
is a continuous and integrable $(\mathcal{F}^{\eps}_t)_{t\geq 0}$ martingale, and if $|\Psi|^2$ is a good test function, its quadratic variation is given by
$$\langle M^{\eps}_{\Psi}\rangle_t=\int_0^t(\mathscr{L}^{\eps}|\Psi|^2-2\Psi\mathscr{L}^{\eps}\Psi)(f^{\eps}_s,m^{\eps}_s)\,\dd s.$$
\end{prop}
\begin{proof}
This is classical, we use the same kind of ideas and follow the proof of \cite[Proposition 6]{arnaudjulien} and \cite[Appendix 6.9]{fgps}.
\end{proof}

\section{The limit generator}

In this section, we study the limit of the generator $\mathscr{L}^{\eps}$ when $\eps\to 0$. The limit generator $\mathscr{L}$ will characterize the limit stochastic fluid equation.

\subsection{Formal derivation of the corrections}\label{choicetestfunctions}
\noindent To derive the diffusive limiting equation, one has to study the limit as $\eps$ goes to $0$ of quantities of the form $\mathscr{L}^{\eps}\Psi$ where $\Psi$ is a good test function. From now on, we choose a specific form for the test functions that we keep thorough the paper. We take $\varphi$ in the Schwartz space $\mathcal{S}(\R^d)$ and we set 
\begin{equation}\label{Psi}
\Psi(f,n):=(f,\varphi F)
\end{equation}
It is clear that $\Psi$ is a good test function. Remember that, when $\eps\to 0$, we will obtain a fluid limit equation verified by the macroscopic quantity $\rho F$; the test function $\Psi$ takes this point in consideration since $\Psi(f,n)=\Psi(f)=\Psi(\rho F)$. In the sequel, we will show that the knowledge of the limits $\mathscr{L}^{\eps}\Psi$ and $\mathscr{L}^{\eps}|\Psi|^2$ as $\eps$ goes to $0$ where $\Psi$ is defined as $(\ref{Psi})$ is sufficient to obtain our result. Nevertheless, we now have to correct $\Psi$ and $|\Psi|^2$ so as to obtain non-singular limits. Here, we show formally how we correct $\Psi$ (the formal work on $|\Psi|^2$ is similar).\\
We search the correction $\Psi^{\eps}$ of $\Psi$. First of all, to correct the deterministic part, we follow the moment method presented in \cite{mellet} so we set $$\Psi^{\eps}(f,n)=(f,\chi^{\eps}F)$$ where $\chi^{\eps}$ solves the auxiliary problem
$$
\chi^{\eps}-\eps v\cdot\nabla_x\chi^{\eps}=\varphi.
$$
Now, to correct the stochastic part, we try an Hilbert expansion method (adapted to our scaling) coupled with the idea of auxiliary equation brought in the moment method so that we complete the definition of $\Psi^{\eps}$ as
$$\Psi^{\eps}(f,n)=(f,\chi^{\eps}F)+\eps^{\frac{\alpha}{2}}(f,\delta^{\eps}F)+\eps^{\alpha}(f,\theta^{\eps}F),$$
where $\delta^{\eps}$ and $\theta^{\eps}$ are to be defined. We then compute, since the first term in the expansion of $\Psi^{\eps}$ does not depend on $n\in E$,
\begin{align}
\mathscr{L}^{\eps}\Psi^{\eps}(f,n)&\label{formal1}=\frac{1}{\eps^{\alpha}}(Lf+\eps Af,\chi^{\eps}F)\\
&\label{formal2}+\frac{1}{\eps^{\frac{\alpha}{2}}}(fn,\chi^{\eps}F)+\frac{1}{\eps^{\frac{\alpha}{2}}}(Lf+\eps Af,\delta^{\eps}F)+\frac{1}{\eps^{\frac{\alpha}{2}}}(f,M\delta^{\eps}F)\\
&\label{formal3}+(fn,\delta^{\eps}F)+(Lf+\eps Af,\theta^{\eps}F)+(f,M\theta^{\eps}F)+\eps^{\frac{\alpha}{2}}(fn,\theta^{\eps}F).
\end{align}
The first term $(\ref{formal1})$ above converges as $\eps$ goes to $0$ to $(-\kappa (-\Delta)^{\frac{\alpha}{2}}f,\varphi F)$, see \cite{mellet}, that is to the infinitesimal generator of the fractional Laplacian applied to $\Psi$: we get the deterministic term of the limiting equation.\\
Since $L$ is auto-adjoint and $A$ is skew-adjoint, the three following terms $(\ref{formal2})$ can be rewritten as
$$\frac{1}{\eps^{\frac{\alpha}{2}}}(f,n\chi^{\eps}F)+\frac{1}{\eps^{\frac{\alpha}{2}}}(f,(L-\eps A+M)(\delta^{\eps}F)).$$
Then we cancel these singular term by choosing $\delta^{\eps}$ such that
$$(L-\eps A+M)(\delta^{\eps}F)=-n\chi^{\eps}F.$$
Formally, this equation can be solved with the resolvent operator of $L-\eps A+M$ so that we have
$$\delta^{\eps}(x,v,n)F(v)=\int_0^{+\infty}e^{tM}g(t,n\chi^{\eps}F)(x+\eps vt,v)\,\dd t.$$
With this expression of $\delta^{\eps}F$ and since $\chi^{\eps}\to\varphi$ as $\eps\to 0$, see \cite{mellet}, we have that $\delta^{\eps}F$ converges to $-M^{-1}I(n)\varphi F$ when $\eps\to 0$. So, neglecting an error term, we can suppose that $(\ref{formal3})$ writes
$$(f,-nM^{-1}I(n)\varphi F)+(Lf+\eps Af,\theta^{\eps}F)+(f,M\theta^{\eps}F)+\eps^{\frac{\alpha}{2}}(fn,\theta^{\eps}F).$$
Note that, for now, the limit of $\mathscr{L}^{\eps}\Psi^{\eps}$ as $\eps$ goes to $0$ does depend on $n\in E$. Since the expected limit is $\mathscr{L}\Psi$ where $\Psi$ does not depend on $n$, we have to correct once again the remaining terms to break the dependence with respect to $n$ of the limit. The right way to do so, given the mixing properties of the operator $M$, is to subtract the mean value: we write $(\ref{formal3})$ as
$$(f,-H\varphi F)+(f,\theta(n)\varphi F)+(Lf+\eps Af,\theta^{\eps}F)+(f,M\theta^{\eps}F)+\eps^{\frac{\alpha}{2}}(fn,\theta^{\eps}F),$$
where $H$ and $\theta$ are respectively defined in $(\ref{defH})$ and $(\ref{deftheta})$. Now, we choose $\theta^{\eps}$ so that $$(L-\eps A+M)(\theta^{\eps}F)=-\theta(n)\varphi F,$$
so that $(\ref{formal3})$ becomes
$$(f,-H\varphi F)+\eps^{\frac{\alpha}{2}}(fn,\theta^{\eps}F);$$
it allows us to conclude that $\mathscr{L}^{\eps}\Psi^{\eps}$ converges to $\mathscr{L}\Psi$ as $\eps\to 0$ where $\mathscr{L}$ is the infinitesimal generator of the equation $(\ref{stochasticeq})$ (note that $D^2\Psi\equiv 0$ so that no stochastic appears here). \\
As we said previously, the same kind of work can be done to correct $|\Psi|^2$. In the following subsections, we define rigorously the corrections of $\Psi$ and $|\Psi|^2$.

\subsection{Preliminaries to the deterministic correction}
\noindent As it is said above, we use the moment method presented in \cite{mellet} to correct the deterministic part of the equation $(\ref{eq})$. Let $\chi^{\eps}$ be the solution of the auxiliary problem
\begin{equation}\label{aux}
\chi^{\eps}-\eps v\cdot\nabla_x\chi^{\eps}=\varphi.
\end{equation}

\noindent We recall, see \cite{mellet}, that the solution of $(\ref{aux})$ is given by 
\begin{equation}\label{defchieps}
\chi^{\eps}(x,v)=\int_0^{+\infty}e^{-t}\varphi(x+\eps vt)\,\dd t,\quad x\in\R^d,\;v\in V.
\end{equation}
We now detail few results on $\chi^{\eps}$.

\begin{prop}\label{chi}The function $\chi^{\eps}F$ is in $\LL$ with
\begin{equation}\label{chinorme}\|\chi^{\eps}F\|\leq \|\varphi\|_{L^2_x}.\end{equation} 
Furthermore, for any $\lambda>0$, we have the following estimate: 
\begin{equation}\label{chiestimate}
\|(\chi^{\eps}-\varphi) F\|^2\lesssim C_{\lambda}^2\eps^2\|\nabla_x\varphi\|^2_{L^2_x}+\|\varphi\|^2_{L^2_x}\lambda^2.
\end{equation}
\end{prop}
\begin{proof}See Appendix \hyperref[appendix]{A}. $\Box$
\end{proof}

\noindent In the two following lemmas, we study in detail the convergence to the fractional Laplace operator. We recall that $\kappa$ has been defined by $(\ref{defkappa})$.
\begin{lemma}\label{lemmelaplacienL2}
For any $\lambda>0$, we have the following estimate:
\begin{equation}\label{laplacienL2}
\left\|\eps^{-\alpha}\int_V[\chi^{\eps}(\cdot,v)-\varphi(\cdot)]F(v)\dd v+\kappa(-\Delta)^{\frac{\alpha}{2}}\varphi\right\|_{L^2_x}\!\!\!\!\!\!\!\lesssim(\Lambda(\eps)+\lambda)(\|\varphi\|_{L^2_x}+\|D^2\varphi\|_{L^2_x}),
\end{equation}
for a certain function $\Lambda$, which only depen\dd s on $\eps$, such that $\Lambda(\eps)\to 0$ when $\eps\to 0$.
\end{lemma}
\begin{proof}See Appendix \hyperref[appendix]{A}. $\Box$
\end{proof}

\begin{lemma}\label{lemmeLaplacien}For any $\lambda>0$, we have the following estimate:
\begin{equation}\label{Laplacien}
\left|\eps^{-\alpha}(\eps Af+Lf,\chi^{\eps}F)+(\kappa(-\Delta)^{\frac{\alpha}{2}}f,\varphi F)\right|\lesssim(\Lambda(\eps)+\lambda)\|f\|(\|\varphi\|_{L^2_x}+\|D^2\varphi\|_{L^2_x}),
\end{equation}
for a certain function $\Lambda$, which only depends on $\eps$, such that $\Lambda(\eps)\to 0$ when $\eps\to 0$.
\end{lemma}	 
\begin{proof}See Appendix \hyperref[appendix]{A}. $\Box$
\end{proof}

\subsection{Preliminaries to the stochastic corrections}
\subsubsection{The corrector $\delta^{\eps}$}

\noindent We recall that $g(t,\cdot)$ denotes the semi-group generated by the operator $L$ on $\LL$ and that the function $\chi^{\eps}$ has been defined in $(\ref{aux})$. Then, we define the function $\delta^{\eps}:\R^d\times V\times E\to \R$ by
$$\delta^{\eps}(x,v,n)F(v):=\int_0^{+\infty}e^{tM}g(t,n\chi^{\eps}F)(x+\eps vt,v)\,\dd t,$$
and we give here some properties of $\delta^{\eps}$. We recall that the test function $\varphi$ has been fixed in Section $\ref{choicetestfunctions}$.
\begin{prop}\label{delta}The function $\delta^{\eps}F$ belongs to $\mathcal{B}(E,\LL)$ with
\begin{equation}\label{deltanorme}
\|\delta^{\eps}F\|_{\mathcal{B}(E,\LL)}\lesssim \|\varphi \|_{L^2_x}.
\end{equation} It satisfies
\begin{equation}\label{deltaeq}
(L-\eps A+M)(\delta^{\eps}F)=-n\chi^{\eps}F,
\end{equation}
with 
\begin{equation}\label{deltanormeM}
\|M\delta^{\eps}F\|_{\mathcal{B}(E,\LL)}\lesssim \|\varphi \|_{L^2_x}.
\end{equation}
Furthermore, for any $\lambda>0$, we have the two following estimates:
\begin{equation}\label{deltaestimate}
\|\delta^{\eps}F+M^{-1}I(n)\varphi F\|_{\mathcal{B}(E,\LL)}\lesssim C_{\lambda}\|\nabla_x\varphi\|_{L^2_x}\eps+\|\varphi\|_{L^2_x}\lambda,
\end{equation}
\begin{equation}\label{deltabonus}
\|M\delta^{\eps}F+n\chi^{\eps} F\|_{\mathcal{B}(E,\LL)}\lesssim C_{\lambda}\|\nabla_x\varphi\|_{L^2_x}\eps+\|\varphi\|_{L^2_x}\lambda.
\end{equation} 
\end{prop}
\begin{proof}
{\em Proof of \eqref{deltanorme}}. The definition of $\delta^{\eps}F$ can be rewritten, thanks to $(\ref{decompositionden})$, as
$$\delta^{\eps}(x,v,n)F(v)=\sum_{i=0}^{+\infty}\int_0^{+\infty}e^{tM}n_i(n)g(t,\eta_i\chi^{\eps}F)(x+\eps vt,v)\,\dd t=:\sum_{i=0}^{+\infty}\alpha_i(x,v,n).$$
Then we fix $i\in\N$ and $n\in E$. We have
\begin{alignat*}{2}
\|\alpha_i(\cdot,\cdot,n)\|^2&=\int_{\R^d}\!\int_V\left(\int_0^{+\infty}e^{tM}n_i(n)g(t,\eta_i\chi^{\eps}F)(x+\eps vt,v)\,\dd t\right)^2\frac{\dd v}{F(v)}\dd x\\
&\leq \int_{\R^d}\!\int_V\left(\int_0^{+\infty}Ke^{-\mu t}|g(t,\eta_i\chi^{\eps}F)|(x+\eps vt,v)\,\dd t\right)^2\frac{\dd v}{F(v)}\dd x\\
&\leq \frac{K^2}{\mu}\int_{\R^d}\!\int_V\int_0^{+\infty}e^{-\mu t}|g(t,\eta_i\chi^{\eps}F)|^2(x+\eps vt,v)\,\dd t\frac{\dd v}{F(v)}\dd x \\
&=\frac{K^2}{\mu}\int_0^{\infty}e^{-\mu t}\|g(t,\eta_i\chi^{\eps}F)\|^2\,\dd t \leq \frac{K^2}{\mu^2}\|\eta_i\chi^{\eps}F\|^2\leq \frac{K^2}{\mu^2}\|\eta_i\|^2_{W^{1,\infty}}\|\varphi F\|^2,
\end{alignat*}
where we used $(\ref{nicentres})$, $(\ref{decroiss})$, $(\ref{borneM1})$, Cauchy-Schwarz inequality, the contraction property of the semigroup $g(t,\cdot)$ $(\ref{contract})$ and finally $(\ref{chinorme})$.
We thus get $$\|\alpha_i\|_{\mathcal{B}(E,\LL)}\leq \frac{K}{\mu}\|\eta_i\|_{W^{1,\infty}}\|\varphi F\|.$$
Since $S=\sum_{i\in\N}\|\eta_i\|_{W^{1,\infty}}<\infty$, we finally deduce that the series defining $\delta^{\eps}F$ converges absolutely in $\mathcal{B}(E,\LL)$ and that
$$\|\delta^{\eps}F\|_{\mathcal{B}(E,\LL)}\lesssim \|\varphi F\|=\|\varphi\|_{L^2_x}.$$
{\em Proof of \eqref{deltaeq}.} Fix $i\in\N$, $\alpha_i$ maps $E$ into $\LL$. We claim that $\alpha_i\in \mathrm{D}(M)$ with, for all $n\in E$,
$$M\alpha_i(x,v,n)=\int_0^{+\infty}Me^{tM}n_i(n)g(t,\eta_i\chi^{\eps}F)(x+\eps vt,v)\,\dd t =:\beta_i(x,v,n)$$
in $\LL$.
Indeed, for $n\in E$, we have
\begin{alignat*}{2}
{}&\int_{\R^d}\!\int_V \left(\frac{e^{hM}\alpha_i(x,v,n)-\alpha_i(x,v,n)}{h}-\beta_i(x,v,n)\right)^2\frac{\dd v}{F(v)}\dd x\\
&=\int_{\R^d}\!\int_V\left(\int_0^{\infty} \left[\frac{e^{(t+h)M}-e^{tM}}{h}-Me^{tM}\right]n_i(n)g(t,\eta_i\chi^{\eps}F)(x+\eps vt,v)\,\dd t\right)^2\frac{\dd v}{F(v)}\dd x\\
&\leq\int_{\R^d}\!\int_V\left(\int_0^{\infty} e^{-\mu t}\left\|\left[\frac{e^{hM}-I}{h}-M\right]n_i\right\|_{\infty}|g(t,\eta_i\chi^{\eps}F)|(x+\eps vt,v)\,\dd t\right)^2\frac{\dd v}{F(v)}\dd x\\
&\leq \frac{1}{\mu^2}\left\|\left[\frac{e^{hM}-I}{h}-M\right]n_i\right\|^2_{\infty}\|\eta_i\|^2_{W^{1,\infty}}\|\varphi F\|^2.
\end{alignat*}
Since by $(\ref{dansdomaine})$, $n\mapsto n_i(n)\in\mathrm{D}(M)$ we deduce that
$$\left\|\frac{e^{hM}\alpha_i-\alpha_i}{h}-\beta_i\right\|_{\mathcal{B}(E,\LL)}\!\!\!\!\!\!\leq \frac{1}{\mu}\left\|\left[\frac{e^{hM}-I}{h}-M\right]n_i\right\|_{\infty}\!\!\!\|\eta_i\|_{W^{1,\infty}}\|\varphi F\| \underset{h\to 0}{\longrightarrow} 0,$$
which is just what we needed. Now, with $(\ref{borneM1})$, we apply Lemma \ref{ehmmoinsiborne} so that we deduce, with the fact that $\sum_{i\in\N}\|\eta_i\|_{W^{1,\infty}}<\infty$ and the dominated convergence theorem, that $\delta^{\eps}F\in\mathrm{D}(M)$ with
$$M[\delta^{\eps}F](x,v,n)=\sum_{i=0}^{\infty}\beta_i(x,v,n),$$
where the series converges absolutely in $\mathcal{B}(E,\LL)$. We fix $i\in\N$, $n\in E$ and $v\in V$. We recall that $\eta_i$ is in $W^{1,\infty}(\R^d)$ and that $\chi^{\eps}$ is defined by $(\ref{defchieps})$ where $\varphi$ is in the Schwartz space $\mathcal{S}(\R^d)$. Then it is easily seen that $\eta_i\chi^{\eps}F$ and $\overline{\eta_i\chi^{\eps}F}$ are in $W^{1,2}(\R^d)$ with respect to $x$. Therefore, since $g(t,\eta_i\chi^{\eps}F)=\overline{\eta_i\chi^{\eps}F}F(1-e^{-t})+\eta_i\chi^{\eps}Fe^{-t}$, we obtain that $h_1:=t\in(0,\infty)\mapsto g(t,\eta_i\chi^{\eps}F)(x+\eps vt,v)$ is in $W^{1,\infty}((0,\infty),L^2_x)$ with $$h_1'(t)(x,v)= Lg(t,\eta_i\chi^{\eps}F)(x+\eps vt,v)+\eps v\cdot\nabla_x g(t,\eta_i\chi^{\eps}F)(x+\eps vt,v),$$ in $L^2_x$.
Furthermore, with $(\ref{decroiss})$, $h_2:=t\in(0,\infty)\mapsto e^{tM}n_i(n)$ is clearly in $W^{1,1}((0,\infty),\R)$ with $h_2'(t)=Me^{tM}n_i(n)$.
We now get by integration by parts
\begin{alignat*}{2}
\beta_i(x,v,n)&=\int_0^{+\infty}Me^{tM}n_i(n)g(t,\eta_i\chi^{\eps}F)(x+\eps vt,v)\,\dd t\\
&\!\!\!\!\!\!\!\!\!\!\!\!\!\!\!\!\!\!\!\!\!\!\!\!\!\!\!\!=\left[e^{tM}n_i(n)g(t,\eta_i\chi^{\eps}F)(x+\eps vt,v)\right]^{\infty}_0-\int_0^{+\infty}e^{tM}n_i(n)\frac{d.}{\dd t}g(t,\eta_i\chi^{\eps}F)(x+\eps vt,v)\,\dd t\\
&=-n_i(n)\eta_i\chi^{\eps}F(x,v)-\int_0^{+\infty}e^{tM}n_i(n)Lg(t,\eta_i\chi^{\eps}F)(x+\eps vt,v)\,\dd t\\
&\,\,\,\,\,\,\,\,\,\,\,\,\,\,-\eps v\cdot\int_0^{+\infty}e^{tM}n_i(n)\nabla_x g(t,\eta_i\chi^{\eps}F)(x+\eps vt,v)\,\dd t,
\end{alignat*}
where all the equalities have to be understood in $L^2_x$.
We easily see that the last two terms of the preceding equality are respectively equal in $L^2_x$ to $-L\alpha_i(x,v,n)$ and $\eps A\alpha_i(x,v,n)$. As a result, we just proved that for all $i\in\N$ and $n\in E$, we have the following equality for almost all $x\in\R^d$ and $v\in V$:
\begin{equation}\label{bs}(L-\eps A+M)\alpha_i(x,v,n)=-n_i(n)\eta_i\chi^{\eps}F(x,v).\end{equation}
Now, the right hand term of the last equality is clearly in $\LL$. Since $\alpha_i$ is in $\LL$, $L\alpha_i\in\LL$; and we proved  above that $M\alpha_i \in\LL$. As a consequence $A\alpha_i$ is in $\LL$ and the preceding equality is valid in $\LL$.
We want to sum over $i\in\N$. We previously proved that we have, in $\mathcal{B}(E,\LL)$, $\sum_{i=0}^{+\infty}M\alpha_i=\sum_{i=0}^{+\infty}\beta_i=M[\delta^{\eps}F].$
Since the series $\sum_{i\in\N}\alpha_i$ converges absolutely in $\mathcal{B}(E,\LL)$ and since $L$ is a bounded operator on $\LL$, we also deduce that we have, in $\mathcal{B}(E,\LL)$, $\sum_{i=0}^{+\infty}L\alpha_i=L[\delta^{\eps}F].$ Since $\sum_{i\in\N}n_i\eta_i$ converges absolutely in $W^{1,\infty}(\R^d)$ to $n$, we obtain that $\sum_{i\in\N}n_i\eta_i\chi^{\eps}F$ converges absolutely in $\mathcal{B}(E,\LL)$ to $n\chi^{\eps}F$. Finally, with $(\ref{bs})$ and the fact that $A$ is a closed operator, we also have, in $\mathcal{B}(E,\LL)$, $\sum_{i=0}^{+\infty}A\alpha_i=A[\delta^{\eps}F].$
Summing $(\ref{bs})$ over $i\in\N$ now gives $(L-\eps A+M)(\delta^{\eps}F)=-n\chi^{\eps}F$.\\
{\em Proof of \eqref{deltanormeM}} We just proved that $M\delta^{\eps}F=\sum_{i=0}^{+\infty}\beta_i$, with
\begin{align*}
\beta_i(x,v,n)&=\int_0^{+\infty}Me^{tM}n_i(n)g(t,\eta_i\chi^{\eps}F)(x+\eps vt,v)\,\dd t\\ &=\int_0^{+\infty}e^{tM}Mn_i(n)g(t,\eta_i\chi^{\eps}F)(x+\eps vt,v)\,\dd t,
\end{align*}
so that we immediately deduce $(\ref{deltanormeM})$ thanks to $(\ref{borneM1})$.\\
{\em Proof of $(\ref{deltaestimate})$.} Let $\lambda >0$. First of all, we point out that $g(t,\eta_i\varphi F)=\eta_i\varphi F$ so that
$$-M^{-1}n_i(n)\eta_i\varphi F(x,v)=\int_0^{\infty}e^{tM}n_i(n)g(t,\eta_i\varphi F)(x,v)\dd t.$$
We can then write, for $i\in\N$ and $n\in E$,
\begin{alignat*}{2}
{}&\|\alpha_i(\cdot,\cdot,n)+M^{-1}n_i(n)\eta_i\varphi F\|^2 \\
&\leq \int_{\R^d}\!\int_V\left(\int_0^{+\infty}e^{tM}n_i(n)g(t,\eta_i(\chi^{\eps}-\varphi)F)(x+\eps vt,v)\,\dd t\right)^2\frac{\dd v}{F(v)}\dd x\\
& +\int_{\R^d}\!\int_V\left(\int_0^{+\infty}e^{tM}n_i(n)\left[g(t,\eta_i\varphi F)(x+\eps vt,v)-g(t,\eta_i\varphi F)(x,v)\right]\,\dd t\right)^2\frac{\dd v}{F(v)}\dd x.
\end{alignat*}
Similarly as the very beginning of the proof, we can bound the first term by 
$$\frac{K^2}{\mu^2}\|\eta_i\|^2_{W^{1,\infty}}\|(\chi^{\eps}-\varphi) F\|^2,$$
and we recall that we have, with $(\ref{chiestimate})$,
$$\|(\chi^{\eps}-\varphi) F\|^2\leq 2C_{\lambda}^2\eps^2\|\nabla_x\varphi\|^2_{L^2_x}+4\|\varphi\|^2_{L^2_x}\lambda^2.$$
For the second term, $B$ say, we write
\begin{alignat*}{2}
B&=\int_{\R^d}\!\int_V\left(\int_0^{+\infty}e^{tM}n_i(n)\left[\eta_i\varphi F(x+\eps vt,v)-\eta_i\varphi F(x,v)\right]\,\dd t\right)^2\frac{\dd v}{F(v)}\dd x.\\
&\leq \frac{K^2}{\mu}\|\eta_i\|^2_{W^{1,\infty}}\int_{\R^d}\!\int_V\int_0^{+\infty}e^{-\mu t}\left[\varphi(x+\eps vt)-\varphi(x)\right]^2\,\dd tF(v)\dd v\dd x.
\end{alignat*}
We can then mimic the proof of Proposition $\ref{chi}$ to get the following bound 
$$\int_{\R^d}\!\int_V\int_0^{+\infty}\!\!\!\!e^{-\mu t}\left[\varphi(x+\eps vt)-\varphi(x)\right]^2\dd tF(v)\dd v\dd x\leq \frac{2C_{\lambda}^2}{\mu^3}\eps^2\|\nabla_x\varphi\|^2_{L^2_x}+\frac{4}{\mu}\|\varphi\|^2_{L^2_x}\lambda^2.$$
To sum up, we just obtained, for $i\in\N$ and $n\in E$,
\begin{alignat*}{2}
\|\alpha_i(\cdot,\cdot,n)+M^{-1}n_i(n)\eta_i\varphi F\|&\lesssim \|\eta_i\|_{W^{1,\infty}}\left(C_{\lambda}^2\|\nabla_x\varphi\|^2_{L^2_x}\eps^2+\|\varphi\|^2_{L^2_x}\lambda^2\right)^{\frac{1}{2}}\\
&\lesssim \|\eta_i\|_{W^{1,\infty}}\left(C_{\lambda}\|\nabla_x\varphi\|_{L^2_x}\eps+\|\varphi\|_{L^2_x}\lambda\right).
\end{alignat*}
We can now sum over $i\in\N$ to obtain,
$$\|\delta^{\eps}F+M^{-1}I(n)\varphi F\|_{\mathcal{B}(E,\LL)}\lesssim C_{\lambda}\|\nabla_x\varphi\|_{L^2_x}\eps+\|\varphi\|_{L^2_x}\lambda,$$
which is the bound expected. \\
{\em Proof of $(\ref{deltabonus})$.} We recall that $M\delta^{\eps}F=\sum_{i=0}^{+\infty}\beta_i$, with $\beta_i$ defined above.
Note that $$n\chi^{\eps}F(x,v)=\sum_{i=0}^{+\infty}\int_0^{+\infty}e^{tM}Mn_i(n)\eta_i\chi^{\eps}F(x,v)\,\dd t,$$
so that we decompose $M\delta^{\eps}(x,v,n)F(v)+n\chi^{\eps}F(x,v)$ into two terms
\begin{multline*}
\sum\limits_{i=0}^{+\infty}\int_0^{+\infty}e^{tM}Mn_i(n)\left[g(t,\eta_i\chi^{\eps}F)(x+\eps vt,v)-g(t,\eta_i\varphi F)(x,v)\right]\,\dd t\\
+\sum\limits_{i=0}^{+\infty}\int_0^{+\infty}e^{tM}Mn_i(n)\left[\eta_i\varphi F)(x,v)-\eta_i\chi^{\eps} F(x,v)\right]\,\dd t.
\end{multline*}
As we have done previously, we can show that the first term is, in $\mathcal{B}(E,\LL)$, \linebreak $\lesssim \left(C_{\lambda}\|\nabla_x\varphi\|_{L^2_x}\eps+\|\varphi\|_{L^2_x}\lambda\right)$. We bound the second term in $\mathcal{B}(E,\LL)$ as $\lesssim \|(\chi^{\eps}-\varphi)F\|$, that is, thanks to $(\ref{chiestimate})$, $\lesssim \left(C_{\lambda}\|\nabla_x\varphi\|_{L^2_x}\eps+\|\varphi\|_{L^2_x}\lambda\right)$. It finally gives the bound expected. This concludes the proof. $\Box$
\end{proof}

\subsubsection{The corrector $\theta^{\eps}$}

\noindent We recall that, for all $n\in E$,
$$ \theta(n)=\int_EnM^{-1}I(n)\dd\nu(n)-nM^{-1}I(n),$$
and that, for $i,j\in\N$, $\theta_{i,j}=\int_En_iM^{-1}n_j\dd\nu-n_iM^{-1}n_j$. Then we define the function $\theta^{\eps}:\R^d\times V\times E\to \R$ by
$$\theta^{\eps}(x,v,n)F(v):=\int_0^{+\infty}e^{tM}g(t,\theta(n)\varphi F)(x+\eps vt,v)\,\dd t,$$
that is,
$$\theta^{\eps}(x,v,n)F(v):=\sum_{i,j=0}^{+\infty}\int_0^{+\infty}e^{tM}\theta_{i,j}(n)g(t,\eta_i\eta_j\varphi F)(x+\eps vt,v)\,\dd t,$$
and, similarly as Proposition \ref{delta}, we obtain the
\begin{prop}\label{theta}The function $\theta^{\eps}F$ belongs to $\mathcal{B}(E,\LL)$ with
\begin{equation}\label{thetanorme}
\|\theta^{\eps}F\|_{\mathcal{B}(E,\LL)}\lesssim \|\varphi\|_{L^2_x}.
\end{equation} It satisfies
\begin{equation}\label{thetaeq}
(L-\eps A+M)(\theta^{\eps}F)=-\theta(n)\varphi F,
\end{equation}
with 
\begin{equation}\label{thetanormeM}
\|M\theta^{\eps}F\|_{\mathcal{B}(E,\LL)}\lesssim \|\varphi \|_{L^2_x}.
\end{equation}
\end{prop}

\subsubsection{The corrector $\zeta^{\eps}$}

We set, for all $(f,n)\in \LL\times E$,
$$ \xi^{\eps}(f,n)=(f,\delta^{\eps}F)n-\int_E(f,\delta^{\eps}F)n\dd\nu(n),$$
and, for $i\in\N$, $\xi^{\eps}_i=(f,\delta^{\eps}F)n_i$. We then define the function $\zeta^{\eps}:\R^d\times V\times\LL\times E\to \R$ by
$$\zeta^{\eps}(x,v,f,n)F(v):=\int_0^{+\infty}e^{tM}g(t,\xi^{\eps}(f,n)\varphi F)(x+\eps vt,v)\,\dd t.$$
Similarly as Proposition $\ref{delta}$, we have the
\begin{prop}\label{zeta}Let $f\in\LL$ be fixed. The function $\zeta^{\eps}F(f)$ belongs to $\mathcal{B}(E,\LL)$ with
\begin{equation}\label{zetanorme}
\|\zeta^{\eps}F(f)\|_{\mathcal{B}(E,\LL)}\lesssim \|f\|\|\varphi \|^2_{L^2_x}.
\end{equation} It satisfies
\begin{equation}\label{zetaeq}
(L-\eps A+M)(\zeta^{\eps}F(f))=-\xi^{\eps}(f,n)\varphi F,
\end{equation}
with 
\begin{equation}\label{zetanormeM}
\|M\zeta^{\eps}F(f)\|_{\mathcal{B}(E,\LL)}\lesssim \|f\|\|\varphi \|^2_{L^2_x}.
\end{equation}
Note that $f\mapsto\zeta^{\eps}F(f)$ is linear. Furthermore, we have for all $f\in\mathrm{D}(A)$,
\begin{equation}\label{zetabonus}
\|\zeta^{\eps}(Lf+\eps Af,\cdot)F\|_{\mathcal{B}(E,\LL)} \lesssim \|f\|\|\varphi\|_{L^2_x}\left(C_{\lambda}\|\nabla_x\varphi\|_{L^2_x}\eps+\|\varphi\|_{L^2_x}\lambda\right).
\end{equation}
\end{prop}
\begin{proof}
We will only prove $(\ref{zetanormeM})$ and $(\ref{zetabonus})$. For the former, we write for $i\in\N$ and $(f,n)\in\LL\times E$,
\begin{align*}
M\xi_i^{\eps}(f,n)&=M(f,\delta^{\eps}(n)F)n_i(n)-\int_EM(f,\delta^{\eps}(n)F)n_i(n)\dd\nu(n)\\
&=\sum_{j=0}^{+\infty}\int_0^{+\infty}\!\!\!\!\!\!\!\!Mn_i(n)e^{tM}n_j(n)(f,g(t,\eta_j\chi^{\eps}F)F)\dd t\\
&\quad -\int_E\sum_{j=0}^{+\infty}\int_0^{+\infty}\!\!\!\!\!\!\!\!Mn_i(n)e^{tM}n_j(n)(f,g(t,\eta_j\chi^{\eps}F)F)\dd t\dd\nu(n),
\end{align*}
so that, with $(\ref{borneM2})$, we have $|M\xi_i^{\eps}(f,n)|\lesssim \|f\|\|\varphi\|_{L^2_x}$. With the definition of $\zeta^{\eps}$, it is now easy to obtain $(\ref{zetanormeM})$.\\
For $(\ref{zetabonus})$, we fix $i\in\N$ and focus on $\xi^{\eps}_i(f,n)$. We have for all $(f,n)\in\mathrm{D}(A)\times E$, 
\begin{align*}
\xi^{\eps}_i(Lf+\eps Af,n)&=(Lf+\eps Af,\delta^{\eps}(n)F)n_i-\int_E(Lf+\eps Af,\delta^{\eps}(n)F)n_i\dd\nu(n) \\
&=(f,(L-\eps A)[\delta^{\eps}(n)F])n_i-\int_E(f,(L-\eps A)[\delta^{\eps}(n)F])n_i\dd\nu(n) \\
&=-(f,M\delta^{\eps}(n)F+n\chi^{\eps}F)n_i+\int_E(f,M\delta^{\eps}(n)F+n\chi^{\eps}F)n_i\dd\nu(n),
\end{align*}
where we used $(\ref{deltaeq})$.
Thanks to $(\ref{deltabonus})$, we thus obtain that, for all $(f,n)\in\mathrm{D}(A)\times E$, 
$$|\xi^{\eps}_i(Lf+\eps Af,n)|\lesssim \|f\|\left(C_{\lambda}\|\nabla_x\varphi\|_{L^2_x}\eps+\|\varphi\|_{L^2_x}\lambda\right).$$
With the expression of $\zeta^{\eps}$, it is now easy to get the required estimate. This concludes the proof. $\Box$
\end{proof}

\subsection{Definition of the corrections}\label{sectioncorrect}
In this section, we precisely define the corrections of the two test functions $\Psi$ and $|\Psi|^2$ that we derived in a formal way in Subsection \ref{choicetestfunctions}. 

\noindent First, we define a deterministic correction by
$$\Psi^{\eps}_*(f,n):=(f,\chi^{\eps}F),\qquad f\in\LL,\;n\in E.$$

\noindent Then, the stochastic corrections for $\Psi$ are defined by, for $(f,n)\in\LL\times E$,
$$
\left\lbrace\begin{aligned}
&\varphi^{\eps}_1(f,n):=(f,\delta^{\eps}(n)F),\\
&\varphi^{\eps}_2(f,n):=(f,\theta^{\eps}(n)F).
\end{aligned}
\right.
$$

\noindent The stochastic corrections for $|\Psi|^2$ are defined by, for $(f,n)\in\LL\times E$,
$$
\left\lbrace\begin{aligned}
&\phi^{\eps}_1(f,n):=2(f,\chi^{\eps}F)(f,\delta^{\eps}(n)F),\\
&\phi^{\eps}_2(f,n):=2(f,\zeta^{\eps}(f,n)F)+2(f,\chi^{\eps}F)(f,\theta^{\eps}(n)F).
\end{aligned}
\right.
$$
\noindent Finally, the corrections $\Psi^{\eps,1}$ and $\Psi^{\eps,2}$ of $\Psi$ and $|\Psi|^2$ are defined by
$$
\left\lbrace\begin{aligned}
&\Psi^{\eps,1}(f,n):=\Psi^{\eps}_*+\eps^{\frac{\alpha}{2}}\varphi^{\eps}_1+\eps^{\alpha}\varphi^{\eps}_2, \\
&\Psi^{\eps,2}(f,n):=|\Psi^{\eps}_*|^2+\eps^{\frac{\alpha}{2}}\phi^{\eps}_1+\eps^{\alpha}\phi^{\eps}_2. \\
\end{aligned}
\right.
$$

\begin{prop}\label{aregoodtest}
For $i=1,2$ and $(f,n)\in\LL\times E$, we have the following estimates:
\begin{alignat}{2}
\label{boundtest}&\varphi^{\eps}_i(f,n)\lesssim \|f\|\|\varphi\|_{L^2_x},\quad &&\phi^{\eps}_i(f,n)\lesssim \|f\|^2\|\varphi\|^2_{L^2_x},\\
\label{boundtestM}& M\varphi^{\eps}_i(f,n)\lesssim \|f\|\|\varphi\|_{L^2_x},\quad && M\phi^{\eps}_i(f,n)\lesssim \|f\|^2\|\varphi\|^2_{L^2_x}.
\end{alignat}
Furthermore, the functions  $\Psi^{\eps}_*$, $|\Psi^{\eps}_*|^2$, $\varphi^{\eps}_1$, $\varphi^{\eps}_2$, $\phi^{\eps}_1$ and $\phi^{\eps}_2$ are good test functions.
Besides, for $(f,n)\in\LL\times E$,
\begin{equation}\label{Dphi2}
|(f,D\phi^{\eps}_2(f,n))|\lesssim \|f\|^2\|\varphi\|^2_{L^2_x}.
\end{equation}
\end{prop}
\begin{proof}
Estimates $(\ref{boundtest})$ and $(\ref{boundtestM})$ are justified by Cauchy Schwarz inequality and $(\ref{deltanorme})$, $(\ref{deltanormeM})$, $(\ref{thetanorme})$, $(\ref{thetanormeM})$, $(\ref{zetanorme})$ and $(\ref{zetanormeM})$. \\
Concerning the fact that all the functions cited above are good test functions, we first note that the case of $\Psi^{\eps}_*$ and $|\Psi^{\eps}_*|^2$ is easy to prove. \\
Let us deal with the case of $\varphi^{\eps}_1$. Conditions $(i)$ and $(iii)$ of Definition \ref{goodtest} are obviously verified. For condition $(ii)$, we have to prove that $D\varphi^{\eps}_1(f,n)\equiv \delta^{\eps}(n)F$ is continuous with respect to $(f,n)\in\LL\times E$, i.e. that $n\mapsto \delta^{\eps}(n)F$ is continuous. We recall that $\delta^{\eps}(x,v,n)F(v)=\sum_{i=0}^{+\infty}\alpha_i(x,v,n)$ in $\mathcal{B}(E,\LL)$ where $$\alpha_i(x,v,n):=\int_0^{+\infty}e^{tM}n_i(n)g(t,\eta_i\chi^{\eps}F)(x+\eps vt,v)\,\dd t.$$
Now, $n\mapsto n_i(n)$ is continuous with Lemma \ref{continuiteenn}, and we thus have thanks to $(\ref{nicentres})$,  $(\ref{decroiss})$, $(\ref{borneM1})$ and the dominated convergence theorem that $n\mapsto\alpha_i(n)$ is continuous. Since the series of the $\alpha_i$ defining $\delta^{\eps}F$ converges in $\mathcal{B}(E,\LL)$, we obtain the continuity of $n\mapsto\delta^{\eps}(n)F$. Furthermore, we can show that $(f,n)\mapsto D\varphi^{\eps}_1(f,n)$ maps bounded sets onto bounded sets thanks to $(\ref{deltanorme})$. So condition $(ii)$ is verified. Similarly, by the continuity of $n\mapsto Mn_i(n)$ (Lemma \ref{continuiteenn}) and by $(\ref{boundtestM})$, we prove that condition $(iv)$ is verified. \\
Similarly, we can prove that $\varphi^{\eps}_2$, $\phi^{\eps}_1$ and $\phi^{\eps}_2$ are good test functions. \\
Finally, since $\zeta^{\eps}(f,n)$ is linear in $f$, for $(f,n)\in\LL\times E$,
$$D\phi^{\eps}_2(f,n)(f)=4(f,\zeta^{\eps}(f,n)F)+4(f,\chi^{\eps}F)(f,\theta^{\eps}(n)F),$$
so that $(\ref{chinorme})$, $(\ref{deltanorme})$ and $(\ref{zetanorme})$ gives $(\ref{Dphi2})$. $\Box$
\end{proof}

\begin{prop}\label{isgoodtestsquare}
The function $(f,n)\mapsto|\Psi^{\eps,1}|^2(f,n)$ is a good test function. Furthermore, we have, for all $(f,n)\in\LL\times E$, the following bounds:
\begin{equation}\label{boundMtest}
\left\{
\begin{aligned}
&|M|\varphi_1^{\eps}|^2(f,n)|\lesssim \|f\|^2\|\varphi\|^2_{L^2_x}, \\
&|M[\varphi_1^{\eps}\varphi_2^{\eps}](f,n)|\lesssim \|f\|^2\|\varphi\|^2_{L^2_x}, \\
&|M|\varphi_2^{\eps}|^2(f,n)|\lesssim \|f\|^2\|\varphi\|^2_{L^2_x},
\end{aligned}
\right.
\end{equation}
and 
\begin{equation}\label{boundMMtest}
\eps^{-\alpha}|M|\Psi^{\eps,1}|^2-2\Psi^{\eps,1}M\Psi^{\eps,1}|\lesssim \|f\|^2\|\varphi\|^2_{L^2_x}.
\end{equation}
\end{prop}
\begin{proof}
In  the expression of $|\Psi^{\eps,1}|^2$, since $\Psi_*^{\eps}$, $\varphi_1^{\eps}$ and $\varphi_2^{\eps}$ are good test functions by Proposition \ref{aregoodtest}, it is easy to prove that $|\Psi_*^{\eps}|^2$, $\Psi_*^{\eps}\varphi_1^{\eps}$ and $\Psi_*^{\eps}\varphi_2^{\eps}$ are also good test functions. It remains to focus on the cases of $|\varphi_1^{\eps}|^2$, $\varphi_1^{\eps}\varphi_2^{\eps}$ and $|\varphi_2^{\eps}|^2$.
We only show the case of $|\varphi_1^{\eps}|^2$ since the others are proved similarly. \\
First, note that point $(i)$ of Definition \ref{goodtest} is clearly verified by $|\varphi_1^{\eps}|^2$ with $D|\varphi_1^{\eps}|^2(f,n)(h)=2(f,\delta^{\eps}(n)F)(h,\delta^{\eps}(n)F)$ and this function of $(f,n)$ maps bounded sets onto bounded sets (thanks to $(\ref{deltanorme})$) and is continuous (is it linear in $f$ and continuous in $n$ since $n\mapsto \delta^{\eps}(n)F$ is continuous, see the proof of Proposition \ref{aregoodtest}). Then we write
\begin{align*}
|\varphi_1^{\eps}|^2(f,n) &= (f,\delta^{\eps}(n)F)^2 = \left(\sum_{i=0}^{+\infty}\int_0^{+\infty}\!\!\!\!\!\!\!\!e^{tM}n_i(n)(f,g(t,\eta_i\chi^{\eps}F)F)\dd t\right)^2 \\
&= \sum_{i,j}\int_0^{\infty}\!\!\!\!\int_0^{\infty}\!\!\!e^{tM}n_i(n)e^{sM}n_j(n)(f,g(t,\eta_i\chi^{\eps}F)F)(f,g(s,\eta_j\chi^{\eps}F)F)\dd t\dd s,
\end{align*}
so that, with $(\ref{dansdomaine})$, $(\ref{borneM2})$ and Lemma \ref{ehmmoinsiborne}, we can mimic the proof of Proposition \ref{delta} to show that $|\varphi_1^{\eps}|^2\in\mathrm{D}(M)$ with
$$
M|\varphi_1^{\eps}|^2(f,n) = \sum_{i,j}\int_0^{\infty}\!\!\!\!\int_0^{\infty}\!\!\!M[e^{tM}n_ie^{sM}n_j](n)(f,g(t,\eta_i\chi^{\eps}F)F)(f,g(s,\eta_j\chi^{\eps}F)F)\dd t\dd s.
$$
Furthermore, with $(\ref{borneM2})$, $(f,n)\mapsto M|\varphi_1^{\eps}|^2(f,n)$ maps bounded sets onto bounded sets (it gives the first bound of $(\ref{boundMtest})$); with $(\ref{continuiteenn})$, $(\ref{borneM2})$ and the dominated convergence theorem, it is continuous with respect to $n$. Since it is linear in $f$ and maps bounded sets onto bounded sets, it is continuous with respect to $(f,n)$. \\
To sum up, we proved that $|\varphi_1^{\eps}|^2(f,n)$ verifies points $(ii)$, $(iii)$ and $(iv)$ of Definition \ref{goodtest}. Finally, we obtain $(\ref{boundMMtest})$ thanks to $(\ref{boundtest})$, $(\ref{boundtestM})$ and $(\ref{boundMtest})$. $\Box$

\end{proof}
\subsection{Convergence to the limit generator}
We first define the limit generator $\mathscr{L}$. For $\psi = \Psi$ or $\psi=|\Psi|^2$, and all $\rho\in L^2(\R^d)$, we set
\begin{multline*}
\mathscr{L}\psi(\rho):=(\rho F,-\kappa(-\Delta)^{\frac{\alpha}{2}}D\psi(\rho F))-\int_E(\rho F nM^{-1}I(n),D\psi(\rho F))\dd\nu(n)\\
-\int_ED^2\psi(\rho F)(\rho Fn,\rho F M^{-1}I(n))\dd\nu(n),
\end{multline*}
and one can easily verify that it is well defined. Then, we state the two results of convergence.
\begin{prop}\label{ordre1}If $(f,n)\in\mathrm{D}(A)\times E$, for any $\lambda>0$, we have the following estimate:
\begin{multline}\label{ordre1dur}
\left|\mathscr{L}^{\eps}\Psi^{\eps,1}(f,n)-\mathscr{L}\Psi(\rho)\right|\lesssim\|f\|\left[\Lambda(\eps)(\|\varphi\|_{L^2_x}+\|D^2\varphi\|_{L^2_x})+C_{\lambda}\|\nabla_x \varphi\|_{L^2_x}\eps\right.\\
\left.+\|\varphi\|_{L^2_x}\eps^{\frac{\alpha}{2}}+(\|\varphi\|_{L^2_x}+\|D^2\varphi\|_{L^2_x})\lambda\right].
\end{multline}
We can also write the right-hand side of the previous bound as
\begin{equation}\label{ordre1simple}
\|f\|(\Lambda(\eps)C_{\varphi,\lambda}+C_{\varphi}\lambda),
\end{equation}
where in each case $\Lambda$ stands for a function which only depends on $\eps$ such that $\Lambda(\eps)\to 0$ when $\eps\to 0$.
\end{prop}
\begin{proof}
We recall that, thanks to Proposition \ref{aregoodtest}, $\Psi^{\eps}_*$, $\varphi_1^{\eps}$ and $\varphi_2^{\eps}$ are good test functions. Then, we compute:
$$\mathscr{L}^{\eps}\Psi^{\eps}_*(f,n)=\frac{1}{\eps^{\alpha}}(Lf+\eps Af,\chi^{\eps}F)+\frac{1}{\eps^{\frac{\alpha}{2}}}(fn,\chi^{\eps}F),$$
where we used the fact that $M\Psi^{\eps}_*(f,n)=0$ since $\Psi^{\eps}_*$ does not depend on $n$. We also have 
\begin{align*}\eps^{\frac{\alpha}{2}}\mathscr{L}^{\eps}\varphi_1^{\eps}(f,n)&=\frac{1}{\eps^{\frac{\alpha}{2}}}(Lf+\eps Af,\delta^{\eps}(n)F)+(fn,\delta^{\eps}(n)F)+\frac{1}{\eps^{\frac{\alpha}{2}}}(f,M\delta^{\eps}(n)F)\\
&=\frac{1}{\eps^{\frac{\alpha}{2}}}(f,(L-\eps A+M)[\delta^{\eps}(n)F])+(fn,\delta^{\eps}(n)F),
\end{align*}
where we used the fact that $L$ (resp. $A$) is auto-adjoint (resp. skew-adjoint) and due to the equation verified by $\delta^{\eps}F$ $(\ref{deltaeq})$, we are led to
$$\eps^{\frac{\alpha}{2}}\mathscr{L}^{\eps}\varphi_1^{\eps}(f,n)=-\frac{1}{\eps^{\frac{\alpha}{2}}}(fn,\chi^{\eps}F)+(fn,\delta^{\eps}(n)F).$$
Furthermore, we have
$$\eps^{\alpha}\mathscr{L}^{\eps}\varphi_2^{\eps}(f,n)=(f,(L-\eps A+M)[\theta^{\eps}(n)F])+\eps^{\frac{\alpha}{2}}(fn,\theta^{\eps}(n)F),$$
that we rewrite, thanks to the equation verified by $\theta^{\eps}F$ $(\ref{thetaeq})$, as
$$\eps^{\alpha}\mathscr{L}^{\eps}\varphi_2^{\eps}(f,n)=-(f,\theta(n)\varphi F)+\eps^{\frac{\alpha}{2}}(fn,\theta^{\eps}(n)F).$$
To sum up, $\mathscr{L}^{\eps}\Psi^{\eps,1}(f,n)=\mathscr{L}^{\eps}\Psi^{\eps}_*(f,n)+\eps^{\frac{\alpha}{2}}\mathscr{L}^{\eps}\varphi_1^{\eps}(f,n)+\eps^{\alpha}\mathscr{L}^{\eps}\varphi_2^{\eps}(f,n)$, hence 
\begin{align*}
\mathscr{L}^{\eps}\Psi^{\eps,1}(f,n)&=\frac{1}{\eps^{\alpha}}(Lf+\eps Af,\chi^{\eps}F)+(fn,\delta^{\eps}(n)F)-(f,\theta(n)\varphi F)+\eps^{\frac{\alpha}{2}}(fn,\theta^{\eps}(n)F)\\
&=\frac{1}{\eps^{\alpha}}(\eps Af+Lf,\chi^{\eps}F)-\int_E(fnM^{-1}I(n),\varphi F)\dd\nu(n)\\
&\qquad\qquad\qquad +(fn,(\delta^{\eps}(n)F+M^{-1}I(n)\varphi F))+\eps^{\frac{\alpha}{2}}(fn,\theta^{\eps}(n)F).
\end{align*}
We point out that $D^2\Psi(f)\equiv 0$ and $(f\psi_1,\psi_2 F)=(\rho F \psi_1,\psi_2 F)$ if $\psi_1$ and $\psi_2$ do not depend on $v\in V$ so that we have
\begin{multline*}
|\mathscr{L}^{\eps}\Psi^{\eps}(f,n)-\mathscr{L}\Psi(\rho)|\leq |\eps^{-\alpha}(\eps Af+Lf,\chi^{\eps}F)+(\kappa(-\Delta)^{\frac{\alpha}{2}}f,\varphi F)|\\+|(fn,(\delta^{\eps}(n)F+M^{-1}I(n)\varphi F))|+\eps^{\frac{\alpha}{2}}|(fn,\theta^{\eps}(n)F)|.
\end{multline*}
We recall that, for all $n\in E$, $\|n\|_{W^{1,\infty}}\lesssim 1$ so that 
$$\left\{\begin{array}{l}
|(fn,(\delta^{\eps}(n)F+M^{-1}I(n)\varphi F))|\lesssim \|f\|\|\delta^{\eps}F+M^{-1}I\varphi F\|_{\mathcal{B}(E,\LL)},\\
|(fn,\theta^{\eps}(n)F)|\lesssim \|f\|\|\theta^{\eps}F\|_{\mathcal{B}(E,\LL)}.
\end{array}\right.$$
Then the bounds $(\ref{Laplacien})$, $(\ref{deltaestimate})$ and $(\ref{thetanorme})$ immediately give the result; this concludes the proof. $\Box$
\end{proof}
\begin{prop}\label{ordre2}If $(f,n)\in\mathrm{D}(A)\times E$, for any $\lambda>0$, we have the following estimate:
$$|\mathscr{L}^{\eps}\Psi^{\eps,2}(f,n)-\mathscr{L}|\Psi|^2(\rho)|\lesssim \Lambda(\eps)C_{\varphi,\lambda}\|f\|^2+C_{\varphi}\|f\|^2\lambda,$$
for a certain function $\Lambda$, which only depends on $\eps$, such that $\Lambda(\eps)\to 0$ when $\eps\to 0$.
\end{prop}
\begin{proof}
We recall that, thanks to Proposition \ref{aregoodtest}, $|\Psi^{\eps}_*|^2$, $\phi_1^{\eps}$ and $\phi_2^{\eps}$ are good test functions. Then, we compute:
$$\mathscr{L}^{\eps}|\Psi^{\eps}_*|^2(f,n)=\frac{2}{\eps^{\alpha}}(L+\eps Af,\chi^{\eps}F)(f,\chi^{\eps}F)+\frac{2}{\eps^{\frac{\alpha}{2}}}(fn,\chi^{\eps}F)(f,\chi^{\eps}F),$$
where we used the fact that $M|\Psi^{\eps}_*|^2(f,n)=0$ since $\Psi^{\eps}_*$ does not depend on $n$. We also have, with the fact that $D\varphi_1(f)(h)=2(h,\chi^{\eps}F)(f,\delta^{\eps}(n)F)+2(h,\delta^{\eps}(n)F)(f,\chi^{\eps}F)$, 
\begin{align*}\eps^{\frac{\alpha}{2}}\mathscr{L}^{\eps}\phi_1^{\eps}(f,n)&=\frac{2}{\eps^{\frac{\alpha}{2}}}(L+\eps Af,\chi^{\eps}F)(f,\delta^{\eps}(n)F)+\frac{2}{\eps^{\frac{\alpha}{2}}}(L+\eps Af,\delta^{\eps}(n)F)(f,\chi^{\eps}F)\\
&\!\!\!\!\!\!\!\!\!\!\!\!\!\!\!\!\!\!\!\!\!\!\!\!\!+2(fn,\chi^{\eps}F)(f,\delta^{\eps}(n)F)+2(fn,\delta^{\eps}(n)F)(f,\chi^{\eps}F)+\frac{2}{\eps^{\frac{\alpha}{2}}}(f,M\delta^{\eps}(n)F)(f,\chi^{\eps}F)\\
&\!\!\!\!\!\!\!\!\!\!\!\!\!\!\!=\frac{2}{\eps^{\frac{\alpha}{2}}}(L+\eps Af,\chi^{\eps}F)(f,\delta^{\eps}(n)F)+\frac{2}{\eps^{\frac{\alpha}{2}}}(f,(L-\eps A+M)[\delta^{\eps}(n)F])(f,\chi^{\eps}F)\\
&\qquad +2(fn,\chi^{\eps}F)(f,\delta^{\eps}(n)F)+2(fn,\delta^{\eps}(n)F)(f,\chi^{\eps}F).
\end{align*}
Thanks to the equation satisfied by $\delta^{\eps}F$ $(\ref{deltaeq})$, we finally get
\begin{multline*}\eps^{\frac{\alpha}{2}}\mathscr{L}^{\eps}\phi^{\eps}_1(f,n)=\frac{2}{\eps^{\frac{\alpha}{2}}}(L+\eps Af,\chi^{\eps}F)(f,\delta^{\eps}(n)F)-\frac{2}{\eps^{\frac{\alpha}{2}}}(fn,\chi^{\eps} F)(f,\chi^{\eps}F)\\
+2(fn,\chi^{\eps}F)(f,\delta^{\eps}(n)F)+2(fn,\delta^{\eps}(n)F)(f,\chi^{\eps}F).
\end{multline*}
Besides, we have
\begin{multline*}\eps^{\alpha}\mathscr{L}^{\eps}\phi_2^{\eps}(f,n)=2(f,(L-\eps A+M)[\zeta^{\eps}(f,n)F])+2(f,(L-\eps A+M)[\theta^{\eps}(n)F])(f,\chi^{\eps}F)\\
+2(Lf+\eps Af,\chi^{\eps}F)(f,\theta^{\eps}(n)F)+2(f,\zeta^{\eps}(Lf+\eps Af,n)F)+\eps^{\frac{\alpha}{2}}(fn,D\phi^{\eps}_2(f,n)),
\end{multline*}
that is, due to equations verified by $\theta^{\eps}F$ and $\zeta^{\eps}F$ $(\ref{thetaeq})$ and $(\ref{zetaeq})$,
\begin{multline*}\eps^{\alpha}\mathscr{L}^{\eps}\phi^{\eps}_2(f,n)=-2(f,\xi^{\eps}\varphi F)-2(f,\theta(n)\varphi F)(f,\chi^{\eps}F)\\
+2(Lf+\eps Af,\chi^{\eps}F)(f,\theta^{\eps}(n)F)+2(f,\zeta^{\eps}(Lf+\eps Af,n)F)+\eps^{\frac{\alpha}{2}}(fn,D\phi^{\eps}_2(f,n)).
\end{multline*}
To sum up, $\mathscr{L}^{\eps}\Psi^{\eps,2}(f,n)=\mathscr{L}^{\eps}|\Psi^{\eps}_*|^2(f,n)+\eps^{\frac{\alpha}{2}}\mathscr{L}^{\eps}\phi^{\eps}_1(f,n)+\eps^{\alpha}\mathscr{L}^{\eps}\phi^{\eps}_2(f,n)$, hence
\begin{multline*}
\mathscr{L}^{\eps}\Psi^{\eps,2}(f,n)=\frac{2}{\eps^{\alpha}}(L+\eps Af,\chi^{\eps}F)(f,\chi^{\eps}F)+\frac{2}{\eps^{\frac{\alpha}{2}}}(L+\eps Af,\chi^{\eps}F)(f,\delta^{\eps}(n)F)\\
+2(fn,\chi^{\eps}F)(f,\delta^{\eps}(n)F)+2(fn,\delta^{\eps}(n)F)(f,\chi^{\eps}F)-2(f,\xi^{\eps}\varphi F)\\
-2(f,\theta(n)\varphi F)(f,\chi^{\eps}F)+2(Lf+\eps Af,\chi^{\eps}F)(f,\theta^{\eps}(n)F)+2(f,\zeta^{\eps}(Lf+\eps Af,n)F)\\+\eps^{\frac{\alpha}{2}}(fn,D\phi^{\eps}_2(f,n)).
\end{multline*}
Now, with the definitions of $\theta$, $\xi$ and the limit generator $\mathscr{L}$, we write the following decomposition 
$\mathscr{L}^{\eps}\Psi^{\eps,2}(f,n)-\mathscr{L}|\Psi|^2(\rho)=\sum_{i=1}^9\tau_i(f,n),$
where
\begin{align*}
\tau_1:&=\frac{2}{\eps^{\alpha}}(L+\eps Af,\chi^{\eps}F)(f,\chi^{\eps}F)-2(-\kappa(-\Delta)^{\frac{\alpha}{2}}f,\varphi F)(f,\varphi F),\\
\tau_2:&=-2\int_E(f,nM^{-1}I(n)\varphi F)(f,(\chi^{\eps}-\varphi)F)\dd\nu(n),\\
\tau_3:&=2\int_E(f,(\delta^{\eps}(n)F+M^{-1}I(n)\varphi F))(fn,\varphi F)\dd\nu(n),
\end{align*}
\begin{alignat*}{2}
\tau_4:&=2(fn,(\delta^{\eps}(n)F+M^{-1}I(n)\varphi F))(f,\chi^{\eps}F), \quad &&\tau_5:=2(f,\delta^{\eps}(n)F)(f,(\chi^{\eps}-\varphi) F),\\
\tau_6:&=\frac{2}{\eps^{\frac{\alpha}{2}}}(Lf+\eps Af,\chi^{\eps}F)(f,\delta^{\eps}(n)F), \quad &&\tau_7:=2(Lf+\eps Af,\chi^{\eps}F)(f,\theta^{\eps}(n)F),\\
\tau_8:&=2(f,\zeta^{\eps}(Lf+\eps Af,n)F), \quad &&\tau_9:=\eps^{\frac{\alpha}{2}}(fn,D\phi^{\eps}_2(f,n)).
\end{alignat*}
To conclude the proof, we are now about to bound every $\tau_i$. For $\tau_1$, we write
\begin{multline*}
\tau_1=\frac{2}{\eps^{\alpha}}(L+\eps Af,\chi^{\eps}F)(f,\chi^{\eps}F)-2(-\kappa(-\Delta)^{\frac{\alpha}{2}}f,\varphi F)(f,\chi^{\eps}F)\\
+2(f,-\kappa(-\Delta)^{\frac{\alpha}{2}}\varphi F)(f,(\chi^{\eps}-\varphi)F),
\end{multline*}
so that, with $(\ref{chinorme})$,
\begin{multline*}
|\tau_1|\lesssim \|f\|\|\varphi\|_{L^2_x}\left|\frac{1}{\eps^{\alpha}}(L+\eps Af,\chi^{\eps}F)+(\kappa(-\Delta)^{\frac{\alpha}{2}}f,\varphi F)\right|\\
+2\|f\|^2\|\kappa(-\Delta)^{\frac{\alpha}{2}}\varphi\|_{L^2_x}\|(\chi^{\eps}-\varphi)F\|,
\end{multline*}
and we use $(\ref{Laplacien})$ and $(\ref{chiestimate})$. Similarly, we bound $\tau_2$ thanks to $(\ref{chiestimate})$, $\tau_3$ thanks to $(\ref{deltaestimate})$, $\tau_4$ thanks to $(\ref{chinorme})$ and $(\ref{deltaestimate})$, $\tau_5$ thanks to $(\ref{deltanorme})$ and $(\ref{chiestimate})$. For $\tau_6$, we write
\begin{multline*}
\tau_6=2\eps^{\frac{\alpha}{2}}\left(\frac{1}{\eps^{\alpha}}(Lf+\eps Af,\chi^{\eps}F)-(-\kappa(-\Delta)^{\frac{\alpha}{2}}f,\varphi F)\right)(f,\delta^{\eps}(n)F)\\+2\eps^{\frac{\alpha}{2}}(f,-\kappa(-\Delta)^{\frac{\alpha}{2}}\varphi F)(f,\delta^{\eps}(n)F),
\end{multline*}
so that, with $(\ref{deltanorme})$,
\begin{multline*}
|\tau_6|\lesssim \eps^{\frac{\alpha}{2}}\|f\|\|\varphi\|_{L^2_x}\left|\frac{1}{\eps^{\alpha}}(L+\eps Af,\chi^{\eps}F)+(\kappa(-\Delta)^{\frac{\alpha}{2}}f,\varphi F)\right|\\+\eps^{\frac{\alpha}{2}}\|f\|^2\|\kappa(-\Delta)^{\frac{\alpha}{2}}\varphi\|_{L^2_x}\|\varphi\|_{L^2_x},
\end{multline*}
and we use $(\ref{Laplacien})$. We handle the case of $\tau_7$ similarly. We bound $\tau_8$ thanks to $(\ref{zetabonus})$, and $\tau_9$ thanks to $(\ref{Dphi2})$. \\
Finally, the combination of the bounds on the $\tau_i$ exactly yields the required result. This concludes the proof. $\Box$ 
\end{proof}

\section{Uniform bound in $\LL$}
In this section, we prove a uniform estimate of the $\LL$ norm of the solution $f^{\eps}$ with respect to $\eps$. To do so, 
we will again use the perturbed test functions method. Thus, let us begin by defining a correction function. Namely, we introduce the function $\iota^{\eps}:\R^d\times V\times E\to \R$ with 
$$\iota^{\eps}(x,v,n):=\sum_{i=0}^{+\infty}\int_0^{+\infty}e^{tM}n_i(n)\eta_i(x+\eps vt)\,\dd t.$$
Similarly as Proposition \ref{delta}, we can prove the
\begin{prop}\label{iota}The function $\iota^{\eps}$ is in $L^{\infty}(\R^d\times V\times E)$ with
\begin{equation}\label{iotanorme}
\|\iota^{\eps}\|_{L^{\infty}(\R^d\times V\times E)}\lesssim 1.
\end{equation}
It satisfies
\begin{equation}\label{iotaeq}
(M-\eps A)(\iota^{\eps})=-n.
\end{equation}
\end{prop}

\begin{prop}
For all $p\geq 1$ and $f_0\in \mathrm{D}(A)$, we have the following bound
\begin{equation}\label{L2bound}
\E\sup\limits_{t\in[0,T]}\|f^{\eps}_t\|^p\, \lesssim \, 1.
\end{equation}
\end{prop}
\begin{proof}
We set, for all $f\in\LL$, $\Theta(f):=\frac{1}{2}\|f\|^2$, which is easily seen to be a good test function. Then,, with the fact that $A$ is skew-adjoint, $(\ref{dissip})$, and the fact that $\Theta$ does not depend on $n\in E$, we get for $f\in\mathrm{D}(A)$ and $n\in E$,
\begin{align*}
\mathscr{L}^{\eps}\Theta(f,n)&=\frac{1}{\eps^{\alpha}}(Lf+\eps Af,f)+\frac{1}{\eps^{\frac{\alpha}{2}}}(fn,f)+\frac{1}{\eps^{\alpha}}M\Theta(f,n)\\
&=-\frac{1}{\eps^{\alpha}}\|Lf\|^2+\frac{1}{\eps^{\frac{\alpha}{2}}}(fn,f).
\end{align*}
The first term has a favourable sign to obtain our bound. The second term is more difficult to control, and we correct $\Theta$ as follows. We set $\phi^{\eps}(f,n)=(f,\iota^{\eps}(n)f)$ and $\Theta^{\eps}(f,n):=\Theta(f,n)+\eps^{\frac{\alpha}{2}}\phi^{\eps}(f,n)$. We can show, with the same method as in the proof of Proposition \ref{aregoodtest}, that $\phi^{\eps}$ is a good test function. We then use integrations by parts and $(\ref{iotaeq})$ to discover 
\begin{align*}\eps^{\frac{\alpha}{2}}\mathscr{L}^{\eps}\phi^{\eps}(f,n)&=\frac{2}{\eps^{\frac{\alpha}{2}}}(Lf,\iota^{\eps}(n)f)+\frac{2}{\eps^{\frac{\alpha}{2}}}(\eps Af,\iota^{\eps}(n)f)+2(fn,\iota^{\eps}(n)f)+\frac{1}{\eps^{\frac{\alpha}{2}}}(f,M\iota^{\eps}(n)f)\\
&=\frac{2}{\eps^{\frac{\alpha}{2}}}(Lf,\iota^{\eps}(n)f)+\frac{1}{\eps^{\frac{\alpha}{2}}}(f,(M-\eps A)[\iota^{\eps}(n)]f)+2(fn,\iota^{\eps}(n)f)\\
&=\frac{2}{\eps^{\frac{\alpha}{2}}}(Lf,\iota^{\eps}(n)f)-\frac{1}{\eps^{\frac{\alpha}{2}}}(fn,f)+2(fn,\iota^{\eps}(n)f).
\end{align*}

\noindent To sum up, since $\mathscr{L}^{\eps}\Theta^{\eps}(f,n)=\mathscr{L}^{\eps}\Theta(f,n)+\eps^{\frac{\alpha}{2}}\mathscr{L}^{\eps}\phi^{\eps}(f,n)$, we have 
$$\mathscr{L}^{\eps}\Theta^{\eps}(f,n)=-\frac{1}{\eps^{\alpha}}\|Lf\|^2+\frac{2}{\eps^{\frac{\alpha}{2}}}(Lf,\iota^{\eps}(n)f)+2(fn,\iota^{\eps}(n)f).$$
We use $(\ref{iotanorme})$ to bound the second term:
\begin{align*}
\frac{2}{\eps^{\frac{\alpha}{2}}}(Lf,\iota^{\eps}(n)f)&\lesssim \frac{1}{\eps^{\frac{\alpha}{2}}}\|Lf\|\|f\|\\
&\leq \frac{\|Lf\|^2}{2\eps^{\alpha}}+\frac{1}{2}\|f\|^2\lesssim\frac{\|Lf\|^2}{2\eps^{\alpha}}+\|f\|^2.
\end{align*}
Besides, note that with $(\ref{iotanorme})$ the third term is $\lesssim\|f\|^2$. Finally we just proved that 
\begin{equation}\label{lepsborne}
|\mathscr{L}^{\eps}\Theta^{\eps}(f,n)|\lesssim \|f\|^2.
\end{equation}

\noindent As in Proposition \ref{gene}, since $\Theta^{\eps}$ is a good test function, we now set,
$$M^{\eps}_{\Theta^{\eps}}(t):=\Theta^{\eps}(f^{\eps}_t,m^{\eps}_t)-\Theta^{\eps}(f^{\eps}_0,m^{\eps}_0)-\int_0^t\mathscr{L}^{\eps}\Theta^{\eps}(f^{\eps}_s,m^{\eps}_s)\,\dd s,$$
which is a continuous and integrable $(\mathcal{F}^{\eps}_t)_{t\geq 0}$ martingale. 
By definition of $\Theta$, $\Theta^{\eps}$ and $M^{\eps}$,
$$\frac{1}{2}\|f^{\eps}_t\|^2=\frac{1}{2}\|f^{\eps}_0\|^2-\eps^{\frac{\alpha}{2}}(\phi^{\eps}(f^{\eps}_t,m^{\eps}_t)-\phi^{\eps}(f^{\eps}_0,m^{\eps}_0))+\int_0^t\mathscr{L}^{\eps}\Theta^{\eps}(f^{\eps}_s,m^{\eps}_s)\,\dd s+M^{\eps}_{\Theta^{\eps}}(t).$$
Since with $(\ref{iotanorme})$ we have $|\phi^{\eps}(f,n)|\lesssim\|f\|^2$, we can write, with $(\ref{lepsborne})$, 
$$\|f^{\eps}_t\|^2\lesssim\|f^{\eps}_0\|^2+\eps^{\frac{\alpha}{2}}\|f^{\eps}_t\|+\int_0^t\|f^{\eps}_s\|^2\,\dd s+\sup\limits_{t\in[0,T]}|M^{\eps}_{\Theta^{\eps}}(t)|,$$
that is, for $\eps$ sufficiently small and by Gronwall Lemma,
\begin{equation}\label{z1}
\|f^{\eps}_t\|^2\lesssim\|f^{\eps}_0\|^2+\sup\limits_{t\in[0,T]}|M^{\eps}_{\Theta^{\eps}}(t)|.
\end{equation}
Furthermore, similarly as Proposition \ref{isgoodtestsquare}, we can show that $|\Theta^{\eps}|^2$ is a good test function, and that 
$$|\mathscr{L}^{\eps}|\Theta^{\eps}|^2-2\Theta^{\eps}\mathscr{L}^{\eps}\Theta^{\eps}|=\eps^{-\alpha}|M|\Theta^{\eps}|^2-2\Theta^{\eps}M\Theta^{\eps}|\lesssim \|f\|^4(1+\Lambda(\eps)),$$
for some function $\Lambda$ which only depends on $\eps$ and such that $\Lambda(\eps)\to 0$ as $\eps\to 0$. In particular, for $\eps$ small enough, 
$$|\mathscr{L}^{\eps}|\Theta^{\eps}|^2-2\Theta^{\eps}\mathscr{L}^{\eps}\Theta^{\eps}|\lesssim \|f\|^4.$$
Besides, with Proposition \ref{gene}, the quadratic variation of $M^{\eps}_{\Theta^{\eps}}(t)$ is given by 
$$\langle M^{\eps}_{\Theta^{\eps}}\rangle_t=\int_0^t(\mathscr{L}^{\eps}|\Theta^{\eps}|^2-2\Theta^{\eps}\mathscr{L}^{\eps}\Theta^{\eps})(f^{\eps}_s,m^{\eps}_s)\,\dd s.$$
As a result, with Burkholder-Davis-Gundy  and Hölder inequalities, we get
\begin{equation}\label{z2}
\E[\sup\limits_{t\in[0,T]}|M^{\eps}_{\Theta^{\eps}}|^p]\lesssim\E[|\langle M^{\eps}_{\Theta^{\eps}}\rangle_T|^{\frac{p}{2}}]\lesssim \int_0^T\E[\|f^{\eps}_s\|^{2p}]\,\dd s.
\end{equation}
By $(\ref{z1})$, we have 
$$\E[\|f^{\eps}_t\|^{2p}]\lesssim\E[\|f^{\eps}_0\|^{2p}]+\E[\sup\limits_{t\in[0,T]}|M^{\eps}_{\Theta^{\eps}}(t)|^p],
$$
so that we get 
$$\E[\|f^{\eps}_T\|^{2p}]\lesssim\E[\|f^{\eps}_0\|^{2p}]+\int_0^T\E[\|f^{\eps}_s\|^{2p}]\,\dd s,
$$
that is, by Gronwall lemma,
$$\E[\|f^{\eps}_T\|^{2p}]\lesssim\E[\|f^{\eps}_0\|^{2p}].$$
This actually holds true for any $t\in[0,T]$. Thus, using $(\ref{z2})$ and then $(\ref{z1})$ gives finally the result. $\Box$
\end{proof}

\section{Summary of the results}
In this section we state the following proposition which sums up all the results obtained above. This will be convenient to handle the tightness and convergence steps. We recall that the corrections $\Psi^{\eps,i}$, $i=1,2$ are defined in Section \ref{sectioncorrect}.
\begin{prop}\label{recap}
Let $f^{\eps}_0\in\mathrm{D}(A)$. For $i=1,2$,
$$M^{\eps}_i(t):=\Psi^{\eps,i}(f^{\eps}_t,m^{\eps}_t)-\Psi^{\eps,i}(f^{\eps}_0,m^{\eps}_0)-\int_0^t\mathscr{L}^{\eps}\Psi^{\eps,i}(f^{\eps}_s,m^{\eps}_s)\,\dd s,\quad t\in[0,T], $$
is a continuous and integrable martingale for the filtration $(\mathcal{F}^{\eps}_t)_{t\geq 0}$ generated by $(m^{\eps}_t,t\geq 0)$. The quadratic variation of $M^{\eps}_1$ is given by
$$\langle M^{\eps}_1\rangle_t=\int_0^t(\mathscr{L}^{\eps}|\Psi^{\eps,1}|^2-2\Psi^{\eps,1}\mathscr{L}^{\eps}\Psi^{\eps,1})(f^{\eps}_s,m^{\eps}_s)\,\dd s, \quad t\in[0,T]$$
and we have, for all $t\in[0,T]$,
\begin{equation}\label{borneducrochet}
|\mathscr{L}^{\eps}|\Psi^{\eps,1}|^2-2\Psi^{\eps,1}\mathscr{L}^{\eps}\Psi^{\eps,1}|(f^{\eps}_t,m^{\eps}_t)\lesssim \sup\limits_{t\in[0,T]} \|f^{\eps}_t\|^2\|\varphi\|^2_{L^2_x}.
\end{equation}
Furthermore, for any $\lambda > 0$, $0\leq s_1\leq\cdots\leq s_n\leq s\leq t$ and $G\in C_b((L^2(\R^d))^n)$,
\begin{equation}\label{convergence1}
\left|\E\left[\left(\Psi(\rho^{\eps}_t F)-\Psi(\rho^{\eps}_sF)-\int_s^t\mathscr{L}\Psi(\rho^{\eps}_{\sigma})\,\dd \sigma\right)G(\rho^{\eps}_{s_1},...,\rho^{\eps}_{s_n})\right]\right|\lesssim \Lambda(\eps)C_{\varphi,\lambda}+C_{\varphi}\lambda,
\end{equation}
\begin{equation}\label{convergence2}
\left|\E\left[\left(|\Psi|^2(\rho^{\eps}_tF)-|\Psi|^2(\rho^{\eps}_sF)-\int_s^t\mathscr{L}|\Psi|^2(\rho^{\eps}_{\sigma})\,\dd \sigma\right)G(\rho^{\eps}_{s_1},...,\rho^{\eps}_{s_n})\right]\right|\lesssim \Lambda(\eps)C_{\varphi,\lambda}+C_{\varphi}\lambda,
\end{equation}
for some function $\Lambda$, which only depends on $\eps$, such that $\Lambda(\eps)\to 0$ when $\eps\to 0$.
Finally, for all $t\in[0,T]$, we have the following estimate:
\begin{equation}\label{grandO}
|\mathscr{L}^{\eps}\Psi^{\eps,1}|(f^{\eps}_t,m^{\eps}_t)\lesssim \sup\limits_{t\in[0,T]}\|f^{\eps}_t\|(\|\varphi\|_{L^2_x}+\|\nabla_x\varphi\|_{L^2_x}+\|D^2\varphi\|_{L^2_x}+\|(-\Delta)^{\frac{\alpha}{2}}\varphi\|_{L^2_x}).
\end{equation}
\end{prop}
\begin{proof}
For $i=1,2$, Proposition \ref{aregoodtest} gives that $\Psi^{\eps,i}$ is a good test function, and it implies, with Proposition \ref{gene}, that $M^{\eps}_i$ is a continuous and integrable martingale. Besides, with Proposition \ref{isgoodtestsquare}, $|\Psi^{\eps,1}|^2$ is a good test function, hence the formula for the quadratic variation of $M^{\eps}_1$.\\
Note that $\mathscr{L}^{\eps}|\Psi^{\eps,1}|^2-2\Psi^{\eps,1}\mathscr{L}^{\eps}\Psi^{\eps,1}=\eps^{-\alpha}(M|\Psi^{\eps,1}|^2-2\Psi^{\eps,1} M\Psi^{\eps,1})$ from which we deduce \eqref{borneducrochet} due to \eqref{boundMMtest}.\\
We continue with the proof of $(\ref{convergence1})$. Observe that $\Psi=\Psi^{\eps,1}+(\Psi-\Psi^{\eps}_*)-\eps^{\frac{\alpha}{2}}\varphi_1^{\eps}-\eps^{\alpha}\varphi_2^{\eps}$ so that we can write
\begin{multline*}
\Psi(f^{\eps}_t)-\Psi(f^{\eps}_s)-\int_s^t\mathscr{L}\Psi(\rho^{\eps}_{\sigma})\,\dd \sigma = M^{\eps}_1(t)-M^{\eps}_1(s)\\
+(\Psi-\Psi^{\eps}_*)(f^{\eps}_t)-(\Psi-\Psi^{\eps}_*)(f^{\eps}_s)-\eps^{\frac{\alpha}{2}}\varphi_1^{\eps}(f^{\eps}_t)-\eps^{\alpha}\varphi_2^{\eps}(f^{\eps}_t)\\
+\eps^{\frac{\alpha}{2}}\varphi_1^{\eps}(f^{\eps}_s)+\eps^{\alpha}\varphi_2^{\eps}(f^{\eps}_s)
+\int_s^t\mathscr{L}^{\eps}\Psi^{\eps,1}(f^{\eps}_{\sigma},m^{\eps}_{\sigma})-\mathscr{L}\Psi(\rho^{\eps}_{\sigma})\,\dd \sigma.
\end{multline*}
Then, we multiply by $G(\rho^{\eps}_{s_1},...,\rho^{\eps}_{s_n})$ and take the expectation. Note that, since $M^{\eps}_1$ is a martingale for the filtration $(\mathcal{F}^{\eps}_t)_{t\geq 0}$ generated by $(m^{\eps}_t,t\geq 0)$, we have $$\E[(M^{\eps}_1(t)-M^{\eps}_1(s))G(J^{-\eta}_r\rho^{\eps}_{s_1},...,J^{-\eta}_r\rho^{\eps}_{s_n})]=0.$$ 
Then, it suffices to use $(\ref{chiestimate})$, $(\ref{boundtest})$, $(\ref{ordre1simple})$, the uniform $\LL$ bound $(\ref{L2bound})$ and $\Psi(f)=\Psi(\rho F)$ to obtain $(\ref{convergence1})$. A similar work can be done to obtain $(\ref{convergence2})$.\\

\noindent It remains to prove $(\ref{grandO})$. We simply write, for $(f,n)\in\mathrm{D}(A)\times E$,
$$|\mathscr{L}^{\eps}\Psi^{\eps,1}(f,n)|\leq|\mathscr{L}^{\eps}\Psi^{\eps,1}(f,n)-\mathscr{L}\Psi(f,n)|+|\mathscr{L}\Psi(f,n)|.$$
We apply $(\ref{ordre1dur})$ with $\eps\leq 1$ and $\lambda=1$ so that
$$|\mathscr{L}^{\eps}\Psi^{\eps,1}(f,n)-\mathscr{L}\Psi(f,n)|\lesssim \|f\|(\|\varphi\|_{L^2_x}+\|\nabla_x\varphi\|_{L^2_x}+\|D^2\varphi\|_{L^2_x}).$$
Since, clearly,
$$|\mathscr{L}\Psi(f,n)|\lesssim\|f\|(\|\kappa(-\Delta)^{\frac{\alpha}{2}}\varphi\|_{L^2_x}+\|\varphi\|_{L^2_x}),$$
the proof is complete. $\Box$
\end{proof}

\section{Tightness}\label{tension}

In this section, in order to be able to take the limit $\eps\to 0$ in law of the family of processes $(\rho^{\eps})_{\eps>0}$, we prove its tightness in an appropriate space, namely $ C([0,T],S^{-\eta}(\R^d))$. Precisely, the result is the following.

\begin{prop}Let $\eta>0$. Then the family $(\rho^{\eps})_{\eps>0}$ is tight in $ C([0,T],S^{-\eta}(\R^d))$.
\end{prop}
\begin{proof}

\noindent We will here specialize the test function $\varphi\in\mathcal{S}(\R^d)$ into the functions $(p_j)_{j\in\N^d}$, which are defined in Section \ref{notations}. So we set, for $j\in\N^d$ and $f\in\LL$, $$\Psi_j(f):=(f,p_jF),$$ and we index by $j\in\N^d$ all the corrections defined in Section \ref{sectioncorrect}.
Thanks to Proposition \ref{recap}, we consider the continuous martingales
$$M^{\eps}_{1,j}(t):=\Psi_{1,j}^{\eps}(f^{\eps}_t,m^{\eps}_t)-\Psi_{1,j}^{\eps}(f^{\eps}_0,m^{\eps}_0)-\int_0^t\mathscr{L}^{\eps}\Psi_{1,j}^{\eps}(f^{\eps}_s,m^{\eps}_s)\,\dd s.$$
We also define, for $j\in\N^d$ and $t\in[0,T]$,
$$\theta_j^{\eps}(t):=\Psi_j(f^{\eps}_0)+\int_0^t\mathscr{L}^{\eps}\Psi_{1,j}^{\eps}(f^{\eps}_s,m^{\eps}_s)\,\dd s+M_{1,j}^{\eps}(t).$$
Note that $$
\theta_j^{\eps}(t)=\Psi_j(f^{\eps}_0)+\Psi_{1,j}^{\eps}(f^{\eps}_t,m^{\eps}_t)-\Psi_{1,j}^{\eps}(f^{\eps}_0,m^{\eps}_0),$$
so that, with Cauchy-Schwarz inequality and $(\ref{boundtest})$, $$|\theta_j^{\eps}(t)|\lesssim\sup\limits_{t\in[0,T]}\|f^{\eps}(t)\|\|p_j\|_{L^2_x}=\sup\limits_{t\in[0,T]}\|f^{\eps}(t)\|.$$
Hence, by the uniform $\LL$ bound $(\ref{L2bound})$,
\begin{equation}\label{bornethetajeps}\E\sup\limits_{t\in [0,T]}\left|\theta_j^{\eps}(t)\right|\lesssim 1.\end{equation}

\noindent We now observe that, for $t\in[0,T]$,
\begin{multline*}
\Psi_j(f^{\eps}_t)-\theta^{\eps}_j(t)=\left[(\Psi_j-\Psi_{*,j}^{\eps})-\eps^{\frac{\alpha}{2}}\varphi_{1,j}^{\eps}-\eps^{\alpha}\varphi_{2,j}^{\eps}\right](f^{\eps}_t,m^{\eps}_t)\\
-\left[(\Psi_j-\Psi_{*,j}^{\eps})-\eps^{\frac{\alpha}{2}}\varphi_{1,j}^{\eps}-\eps^{\alpha}\varphi_{2,j}^{\eps}\right](f^{\eps}_0,m^{\eps}_0),
\end{multline*}
and it gives, with Cauchy-Schwarz inequality, $(\ref{chiestimate})$, $(\ref{boundtest}$), and $(\ref{taillej})$,
\begin{align}
\nonumber\left|\Psi_j(f^{\eps}_t)-\theta^{\eps}_j(t)\right|&\lesssim \sup\limits_{t\in [0,T]}\|f^{\eps}_t\|\|(\chi^{\eps}_j-p_j)F\|+(\eps^{\frac{\alpha}{2}}+\eps^{\alpha})\|f^{\eps}_t\|\|p_j\|_{L^2_x}\\
\nonumber&\leq \sup\limits_{t\in [0,T]}\|f^{\eps}_t\|(C_{\lambda}\eps\|\nabla_x p_j\|_{L^2_x}+\|p_j\|_{L^2_x}\lambda +(\eps^{\frac{\alpha}{2}}+\eps^{\alpha})\|p_j\|_{L^2_x})\\
\label{aaa}&\leq \sup\limits_{t\in [0,T]}\|f^{\eps}_t\|(C_{\lambda}\eps \mu_j^{\frac{1}{2}}+\lambda +\eps^{\frac{\alpha}{2}}+\eps^{\alpha}).
\end{align}

\noindent From now on, we fix $\gamma > d/2+1$. Observe that, by $(\ref{bornethetajeps})$, a.s. and for all $t\in[0,T]$, the series defined by $u^{\eps}_t:=\sum_{j\in\N^d}\theta^{\eps}_j(t)J^{-\gamma}p_j$ converges in $L^2(\R^d)$, which is embedded in $\mathcal{S}'(\R^d)$. We then set $$\theta^{\eps}_t:=J^{\gamma}\sum\limits_{j\in\N^d}\theta^{\eps}_j(t)J^{-\gamma}p_j,$$
which exists a.s. and for all $t\in[0,T]$ in $\mathcal{S}'(\R^d)$.
In fact, we see that a.s. and for all $t\in[0,T]$, $\theta^{\eps}_t$ is in $S^{-\gamma}(\R^d)$. Indeed, 
\begin{align*}
\|\theta^{\eps}_t\|^2_{S^{-\gamma}(\R^d)}&=\|J^{\gamma}u^{\eps}_t\|^2_{S^{-\gamma}(\R^d)}
=\|u^{\eps}_t\|^2_{L^2_x}<\infty.
\end{align*}
We point out that $\Psi_j(f^{\eps}_t)=(\rho^{\eps}_tF,p_jF)=(\rho^{\eps}_t,p_j)_x$ so that 
\begin{align*}
\langle \rho^{\eps}(t)-\theta^{\eps}(t),p_j\rangle&=\Psi_j(f^{\eps}_t)-\langle J^{\gamma}u^{\eps}_t,p_j\rangle
=\Psi_j(f^{\eps}_t)-\langle u^{\eps}_t,J^{\gamma}p_j\rangle\\
&=\Psi_j(f^{\eps}_t)-\langle u^{\eps}_t,p_j\rangle\mu_j^{\gamma}=\Psi_j(f^{\eps}_t)-\theta_j^{\eps}(t).
\end{align*}
By $(\ref{aaa})$, it permits to write, for $t\in[0,T]$,
\begin{align*}
\left\|\rho^{\eps}(t)-\theta^{\eps}(t)\right\|^2_{S^{-\gamma}(\R^d)}&
\lesssim \sum\limits_{j\in\N^d}\mu_j^{-2\gamma}\sup\limits_{t\in [0,T]}\|f^{\eps}_t\|^2(C_{\lambda}\eps^2\mu_j+\lambda^2 +\eps^{\alpha}+\eps^{2\alpha})\\
&\lesssim \sup\limits_{t\in [0,T]}\|f^{\eps}_t\|^2(C_{\lambda}\eps^2+\eps^{\alpha}+\eps^{2\alpha}+\lambda^2)
\end{align*}
where the second bound comes from our choice $\gamma >d/2+1$ (we recall, see Section \ref{notations}, that $\mu_j=2|j|+1$).  Thanks to the uniform $\LL$ bound $(\ref{L2bound})$, it finally leads to the following estimate:
\begin{equation}\label{sontproches}
\E\sup\limits_{t\in [0,T]}\left\|\rho^{\eps}(t)-\theta^{\eps}(t)\right\|_{S^{-\gamma}(\R^d)}\lesssim C_{\lambda}\eps+\eps^{\frac{\alpha}{2}}+\eps^{\alpha}+\lambda.
\end{equation}

\noindent We now fix $\eta>0$. For any $\delta>0$, let $$w(\rho,\delta):=\sup\limits_{|t-s|<\delta}\|\rho(t)-\rho(s)\|_{S^{-\eta}(\R^d)}$$
denote the modulus of continuity of a function $\rho\in C([0,T],S^{-\eta}(\R^d))$. Since the injection $L^2(\R^d)\subset S^{-\eta}(\R^d)$ is compact, and by Ascoli's theorem, the set 
$$K_R:=\left\lbrace\rho\in C([0,T],S^{-\eta}(\R^d)),\; \sup\limits_{t\in[0,T]}\|\rho\|_{L^2(\R^d)}\leq R,\; w(\rho,\delta)<\eps(\delta)\right\rbrace,$$
where $R>0$ and $\eps(\delta)\to 0$ when $\delta\to 0$, is compact in $ C([0,T],S^{-\eta}(\R^d))$. To prove the tightness of $(\rho^{\eps})_{\eps>0}$ in $ C([0,T],S^{-\eta}(\R^d))$, it thus suffices, see \cite{billingsley}, to prove that for all $\sigma>0$, there exists $R>0$ such that \begin{equation}\label{tight1}\PP(\sup\limits_{t\in[0,T]}\|\rho^{\eps}\|_{L^2(\R^d)}>R)<\sigma,
\end{equation}
and
\begin{equation}\label{tight2}\lim\limits_{\delta\to 0}\limsup\limits_{\eps\to 0}\PP(w(\rho^{\eps},\delta)>\sigma)=0.
\end{equation}
By the continuous embedding $L^2(\R^d)\subset S^{-\eta}(\R^d)$ and Markov's inequality, we have
$$\PP(\sup\limits_{t\in[0,T]}\|\rho^{\eps}\|_{L^2(\R^d)}>R)\leq\PP(\sup\limits_{t\in[0,T]}\|f^{\eps}\|_{\LL}>R)\leq\frac{1}{R}\E[\sup\limits_{t\in[0,T]}\|f^{\eps}\|_{\LL}],$$
and it gives $(\ref{tight1})$ thanks to the uniform $\LL$ bound $(\ref{L2bound})$.\\
Similarly, we will deduce $(\ref{tight2})$ by Markov's inequality and a bound on $\E[w(\rho^{\eps},\delta)]$ for $\delta>0$. Actually, by interpolation, the continuous embedding $L^2(\R^d)\subset S^{-\eta}(\R^d)$ and the uniform $\LL$ bound $(\ref{L2bound})$, we have 
$$\E\sup\limits_{|t-s|<\delta}\|\rho(t)-\rho(s)\|_{S^{-\eta^{\flat}}}\leq\E\sup\limits_{|t-s|<\delta}\|\rho(t)-\rho(s)\|^{\upsilon}_{S^{-\eta^{\sharp}}}$$
for a certain $\upsilon>0$ if $\eta^{\sharp}>\eta^{\flat}>0$. As a result, it is indeed sufficient to work with $\eta=\gamma$. In view of $(\ref{sontproches})$, we first want to obtain an estimate of the increments of $\theta^{\eps}$. We have, for $j\in\N^d$ and $0\leq s\leq t\leq T$,
$$\theta^{\eps}_j(t)-\theta^{\eps}_j(s)=\int_s^t\mathscr{L}^{\eps}\Psi_{1,j}^{\eps}(f^{\eps}_{\sigma},m^{\eps}_{\sigma})\,\dd \sigma+M^{\eps}_{1,j}(t)-M^{\eps}_{1,j}(s).$$
By $(\ref{grandO})$ and the uniform $\LL$ bound $(\ref{L2bound})$, we have
$$\E\left|\int_s^t\mathscr{L}^{\eps}\Psi_{1,j}^{\eps}(f^{\eps}_{\sigma},m^{\eps}_{\sigma})\,\dd \sigma\right|^4\lesssim C_j|t-s|^4,$$
where $$C_j:=(\|p_j\|_{L^2_x}+\|\nabla_xp_j\|_{L^2_x}+\|D^2p_j\|_{L^2_x}+\|(-\Delta)^{\frac{\alpha}{2}}p_j\|_{L^2_x}).$$
Furthermore, using Burkholder-Davis-Gundy inequality,
$$\E|M^{\eps}_{1,j}(t)-M^{\eps}_{1,j}(s)|^4\lesssim \E|\langle M^{\eps}_{1,j}\rangle_t-\langle M^{\eps}_{1,j}\rangle_s|^2,$$
and thanks to $(\ref{borneducrochet})$, the uniform $\LL$ bound $(\ref{L2bound})$ and the fact that $\|p_j\|_{L^2_x}=1$, we are led to
$$\E|M^{\eps}_j(t)-M^{\eps}_j(s)|^4\lesssim |t-s|^2.$$
Finally we have $\E|\theta^{\eps}_j(t)-\theta^{\eps}_j(s)|^4\lesssim (1+C_j)|t-s|^2$. Now, note that with $(\ref{taillej})$, 
$C_j\lesssim 1+\sqrt{\mu_j}+\mu_j$. Since we took $\gamma>d/2+1$, we can conclude that 
$$\E\|\theta^{\eps}_t-\theta^{\eps}_s\|_{S^{-\gamma}(\R^d)}^4\lesssim |t-s|^2.$$
It easily follows that, for $\upsilon<1/2$, $\E\|\theta^{\eps}\|_{W^{\upsilon,4}(0,T,S^{-\gamma}(\R^d))}^4\lesssim 1$ so that  by the Sobolev embedding $W^{\upsilon,4}(0,T,S^{-\gamma}(\R^d))\subset  C^{0,\tau}(0,T,S^{-\gamma}(\R^d))$ which holds true whenever $\tau<\upsilon-1/4$, we obtain that $\E w(\theta^{\eps},\delta)\lesssim \delta^{\tau}$ for a certain positive $\tau$.\\
Thus, we deduce, with $(\ref{sontproches})$,
\begin{align*} 
\E w(\rho^{\eps},\delta)& \leq 2\E\sup\limits_{t\in [0,T]}\left\|\rho^{\eps}_t-\theta^{\eps}_t\right\|_{S^{-\gamma}(\R^d)}+\E w(\theta^{\eps},\delta)\\
& \lesssim C_{\lambda}\eps+\eps^{\frac{\alpha}{2}}+\eps^{\alpha}+\lambda+ \delta^{\tau}.
\end{align*}
To conclude, we then have
\begin{align*}
\lim\limits_{\delta\to 0}\limsup\limits_{\eps\to 0}\PP(w(\rho^{\eps},\delta)>\sigma)&\leq\lim\limits_{\delta\to 0}\limsup\limits_{\eps\to 0}\sigma^{-1}\E w(\rho^{\eps},\delta) \lesssim \sigma^{-1}\lambda,
\end{align*}
and since $\lambda>0$ was arbitrary, we just proved $(\ref{tight2})$. This concludes the proof. $\Box$
\end{proof}

\section{Convergence}

In this section, we conclude the proof of Theorem \ref{mainresult}. The idea is now, by the tightness result proved above and Prokhorov's Theorem, to take a subsequence of $(\rho^{\eps})_{\eps>0}$ that converges in law to some probability measure. Then we show that this limit probability is actually uniquely determined thanks to the convergences to the limit generator $\mathscr{L}$ proved above.\\

\noindent Let us fix $\eta>0$. By Proposition $\ref{tension}$ and Prokhorov's Theorem, there exist a subsequence of $(\rho^{\eps})_{\eps >0}$, still denoted $(\rho^{\eps})_{\eps >0}$, and a probability measure $P$ on $ C([0,T],S^{-\eta}(\R^d))$ such that
$$P^{\eps}\to P\text{ weakly on }  C([0,T],S^{-\eta}(\R^d)),$$
where $P^{\eps}$ stands for the law of $\rho^{\eps}$. We will now identify the probability measure $P$. Since $ C([0,T],S^{-\eta}(\R^d))$ is separable, we can apply Skohorod representation Theorem \cite{billingsley}, so that there exist a new probability space $(\widetilde{\Omega},\widetilde{\mathcal{F}},\widetilde{\PP})$ and random variables $$\widetilde{\rho^{\eps}},\widetilde{\rho}:\widetilde{\Omega}\to C([0,T],S^{-\eta}(\R^d)),$$ with respective law $P^{\eps}$ and $P$ such that $\widetilde{\rho^{\eps}}\to \widetilde{\rho}$ in $ C([0,T],S^{-\eta}(\R^d))$, $\widetilde{\PP}-$a.s. In the sequel, for the sake of clarity, we do not write any more the tildes.\\

\noindent Let us pass to the limit $\eps\to 0$ in the left-hand side of \eqref{convergence1}, namely in the quantity
$$
\E\left[\left(\Psi(\rho^{\eps}_t F)-\Psi(\rho^{\eps}_sF)-\int_s^t\mathscr{L}\Psi(\rho^{\eps}_{\sigma})\,\dd \sigma\right)G(\rho^{\eps}_{s_1},...,\rho^{\eps}_{s_n})\right]=:\E[\mathcal{A}(\rho^{\eps})].
$$
Without loss of any generality, we may assume that the function $G\in C_b((L^2(\R^d))^n)$ is also continuous on the space $H^{-\eta}$; this is always possible with an approximation argument: it suffices to consider $G_r:=G((\textrm{I}+rJ)^{-\frac{\eta}{2}}\cdot,...,(\textrm{I}+r J)^{-\frac{\eta}{2}}\cdot)$ as $r\to 0$. Then, with the $\PP-$a.s. convergence of $\rho^{\eps}$ to $\rho$ in the space $C([0,T],S^{-\eta}(\R^d))$, we have that
$$\mathcal{A}(\rho^{\eps})\to\mathcal{A}(\rho),\quad \text{a.s.}$$
Furthermore, thanks to the uniform $\LL$ bound \eqref{L2bound} and the boundedness of $G$, $(\mathcal{A}(\rho^{\eps}))_{\eps>0}$ is uniformly integrable since it is bounded in $L^2(\Omega)$, hence
$$\E\mathcal{A}(\rho^{\eps})\to\E\mathcal{A}(\rho).$$
As a consequence, we can now pass to the limit $\eps\to 0$ in \eqref{convergence1} to discover
$$
\left|\E\left[\left(\Psi(\rho_t F)-\Psi(\rho_sF)-\int_s^t\mathscr{L}\Psi(\rho_{\sigma})\,\dd \sigma\right)G(\rho_{s_1},...,\rho_{s_n})\right]\right|\lesssim C_{\varphi}\lambda.
$$
Since this holds true for arbitrary $\lambda > 0$, it yields
\begin{equation}\label{martingale1}
\E\left[\left(\Psi(\rho_t F)-\Psi(\rho_sF)-\int_s^t\mathscr{L}\Psi(\rho_{\sigma})\,\dd \sigma\right)G(\rho_{s_1},...,\rho_{s_n})\right]= 0.
\end{equation}
Similarly, we can pass to the limit $\eps\to 0$ in \eqref{convergence2}; it gives
\begin{equation}\label{martingale2}
\E\left[\left(|\Psi|^2(\rho_tF)-|\Psi|^2(\rho_sF)-\int_s^t\mathscr{L}|\Psi|^2(\rho_{\sigma})\,\dd \sigma\right)G(\rho_{s_1},...,\rho_{s_n})\right]= 0.
\end{equation}
\noindent Since $(\ref{martingale1})$ and $(\ref{martingale2})$ are valid for all $n\in\N$, $s_1\leq...\leq s_n\leq s\leq t \in[0,T]$ and all $G\in C_b((L^2(\R^d))^n)$, we deduce that 
$$N(t):=\Psi(\rho_tF)-\Psi(\rho_0 F)-\int_0^t\mathscr{L}\Psi(\rho_{\sigma})\,\dd \sigma,\quad t\in[0,T],$$
and
$$S(t):=|\Psi|^2(\rho_tF)-|\Psi|^2(\rho_0F)-\int_0^t\mathscr{L}|\Psi|^2(\rho_{\sigma})\,\dd \sigma, \quad  t\in[0,T],$$
are martingales with respect to the filtration generated by $(\rho_s)_{s\in[0,T]}$.
It implies that, see \cite[Appendix 6.9]{fgps}, the quadratic variation of $N$ is given by
$$\langle N\rangle_t=\int_0^t\left[\mathscr{L}|\Psi|^2(\rho_{\sigma})-2\Psi(\rho_{\sigma})\mathscr{L}\Psi(\rho_{\sigma})\right]\,\dd \sigma,\quad t\in[0,T].$$
Furthermore, we have
\begin{align*}
\mathscr{L}|\Psi|^2(\rho_{\sigma})-2\Psi(\rho_{\sigma})\mathscr{L}\Psi(\rho_{\sigma})&=-2\int_E(\rho_{\sigma}n,\varphi)_x(\rho_{\sigma}M^{-1}I(n),\varphi)_x\,\dd\nu(n)\\
&=2\E[\int_0^{\infty}(\rho_{\sigma}m_0,\varphi)_x(\rho_{\sigma}m_t,\varphi)_x\,\dd t]\\
&=\E[\int_{\R}(\rho_{\sigma}m_0,\varphi)_x(\rho_{\sigma}m_t,\varphi)_x\,\dd t]\\
&=\int_{\R^d}\int_{\R^d}\rho_{\sigma}(x)\varphi(x)\rho_{\sigma}(y)\varphi(y)k(x,y)\,\dd x \dd y\\
&=\|\rho_{\sigma}Q^{\frac{1}{2}}\varphi\|^2_{L^2_x}.
\end{align*}
Here, we recall that $\Psi(\rho F)=(\rho F,\varphi F)=(\rho,\varphi)_x$ and that the results above are valid for all $\varphi\in\mathcal{S}(\R^d)$. As a consequence, the martingale $N$ gives us that 
$$M(t):=\rho_t-\rho_0-\int_0^t[-\kappa(-\Delta)^{\frac{\alpha}{2}}\rho_{\sigma}-\frac{1}{2}\rho_{\sigma}H]\,\dd \sigma, \quad t\in[0,T],$$
is a continuous martingale in $L^2(\R^d)$ with respect to the filtration generated by $(\rho_s)_{s\in[0,T]}$ with quadratic variation
$$\langle M\rangle_t=\int_0^t(\rho_{\sigma}Q^{\frac{1}{2}})(\rho_{\sigma}Q^{\frac{1}{2}})^*\,\dd \sigma,\quad t\in[0,T].$$
Thanks to martingale representation Theorem, see \cite[Theorem 8.2]{daprato}, up to a change of probability space, there exists a cylindrical Wiener process $W$ in $L^2(\R^d)$ such that 
$$\rho_t-\rho_0-\int_0^t[-\kappa(-\Delta)^{\frac{\alpha}{2}}\rho_{\sigma}-\frac{1}{2}\rho_{\sigma}H]\,\dd \sigma=\int_0^t\rho_{\sigma}Q^{\frac{1}{2}}\,\dd W_{\sigma}, \quad t\in[0,T].$$
This equality gives that $\rho$ has the law of a weak solution to the equation $(\ref{stochasticeq})$ with paths in $C([0,T],S^{-\eta}(\R^d))$. Since this equation has a unique solution with paths in the space $C([0,T],S^{-\eta}(\R^d))\cap L^{\infty}([0,T],L^2(\R^d))$, and since pathwise uniqueness implies uniqueness in law, we deduce that $P$ is the law of this solution and is uniquely determined. Finally, by the uniqueness of the limit, the whole sequence $(P^{\eps})_{\eps>0}$ converges to $P$ weakly in the space of probability measures on $C([0,T],S^{-\eta}(\R^d))$. $\Box$

\nocite{degond}
\nocite{debouard}

\section*{Appendix A}\label{appendix}
\begin{proof}[Proof of Proposition \ref{chi}]
For the first bound, we write, thanks to Cauchy-Schwarz inequality,
\begin{alignat*}{2}
\|\chi^{\eps}F\|^2&=\int_{\R^d}\!\int_V\left(\int_0^{+\infty}e^{-t}\varphi(x+\eps vt)\,\dd t\right)^2F(v)\dd v\dd x\\
&\leq\int_{\R^d}\!\int_V\int_0^{+\infty}e^{-t}\varphi^2(x+\eps vt)F(v)\,\dd t\dd v\dd x \\
&=\|\varphi\|^2_{L^2_x}\int_V\int_0^{+\infty}e^{-t}F(v)\,\dd t\dd v=\|\varphi\|^2_{L^2_x}.
\end{alignat*}

\noindent To prove the second estimate, we fix $\lambda>0$. Since $F$ is integrable with respect to $v$, we take $C_{\lambda}>0$ such that $\int_{\left\lbrace |v|\geq C_{\lambda}\right\rbrace}F(v)\,\dd v < \lambda^2$. We have 
$$\|(\chi^{\eps}-\varphi)F\|^2=\int_{\R^d}\!\int_V\left(\int_0^{+\infty}e^{-t}[\varphi(x+\eps vt)-\varphi(x)]\,\dd t\right)^2F(v)\dd v\dd x.$$
Then we split the $v$-integral into two terms $\tau_1$ and $\tau_2$:
\begin{alignat*}{2}
\tau_1&:=\int_{\R^d}\int_{|v|\geq C_{\lambda}}\int_0^{+\infty}e^{-z}\left[\varphi(x+\eps v z)-\varphi(x)\right]^2F(v)\,\dd z\dd v\dd x \\
&\leq 2\int_{\R^d}\int_{ |v|\geq C_{\lambda}}\int_0^{+\infty}e^{-z}\left(|\varphi(x+\eps v z)|^2+|\varphi(x)|^2\right)F(v)\,\dd z\dd v\dd x\\
&=4\|\varphi\|_{L^2_x}^2\int_{ |v|\geq C_{\lambda}}\int_0^{+\infty}e^{-z}F(v)\,\dd z\dd v<4\|\varphi\|_{L^2_x}^2\lambda^2 \,\,;
\end{alignat*}

\begin{alignat*}{2}
\tau_2&:=\int_{\R^d}\int_{ |v|\leq C_{\lambda}}\int_0^{+\infty}e^{-z}\left[\varphi(x+\eps v z)-\varphi(x)\right]^2F(v)\,\dd z\dd v\dd x \\
&=\int_{\R^d}\int_{ |v|\leq C_{\lambda}}\int_0^{+\infty}e^{-z}\left(\int_0^1\eps zv\cdot\nabla_x\varphi(x+t\eps z v)\,\dd t\right)^2F(v)\,\dd z\dd v\dd x \\
&\leq\eps^2\int_{\R^d}\int_{|v|\leq C_{\lambda}}\int_0^{+\infty}\int_0^1 e^{-z}z^2|v|^2|\nabla_x\varphi(x+t\eps z v)|^2F(v)\,\dd t\dd z\dd v\dd x \\
&\leq 2\eps^2C_{\lambda}^2\|\nabla_x\varphi\|_{L^2_x}^2,
\end{alignat*}
and this is the result. $\Box$
\end{proof}

\begin{proof}[Proof of Lemma \ref{lemmelaplacienL2}]
We fix $\lambda>0$. Then we choose $C$ such that, for all $|v|\geq C$,
\begin{equation}\label{asymp}
\left|F(v)-\frac{\kappa_0}{|v|^{d+\alpha}}\right|\leq\frac{\lambda\kappa_0}{|v|^{d+\alpha}}.
\end{equation}
Now, we write, for $x\in\R^d$,
\begin{alignat*}{2}
&\eps^{-\alpha}\int_V\int_0^{+\infty}e^{-t}\left[\varphi(x+\eps vt)-\varphi(x)\right]F(v)\,\dd t\dd v\\
&=\eps^{-\alpha}\int_{|v|\leq C}\int_0^{+\infty}e^{-t}\left[\varphi(x+\eps vt)-\varphi(x)\right]F(v)\,\dd t\dd v\\
&+\eps^{-\alpha}\int_{|v|\geq C}\int_0^{+\infty}e^{-t}\left[\varphi(x+\eps vt)-\varphi(x)\right]\frac{\kappa_0}{|v|^{d+\alpha}}\,\dd t\dd v\\
&+\eps^{-\alpha}\int_{|v|\geq C}\int_0^{+\infty}e^{-t}\left[\varphi(x+\eps vt)-\varphi(x)\right]\left[F(v)-\frac{\kappa_0}{|v|^{d+\alpha}}\right]\,\dd t\dd v\\
&=:I_1(x)+I_2(x)+I_3(x).
\end{alignat*}
We begin by bounding $\|I_1\|^2_{L^2_x}$. Since $F(v)=F(-v)$, we rewrite $I_1(x)$ as follows
\begin{alignat*}{2}
I_1(x)&=\eps^{-\alpha}\int_{|v|\leq C}\int_0^{+\infty}e^{-t}\left[\varphi(x+\eps vt)-\varphi(x)-\eps vt\cdot\nabla_x\varphi(x)\right]F(v)\,\dd t\dd v\\
&=\eps^{-\alpha}\int_{|v|\leq C}\int_0^{+\infty}\int_0^1e^{-t}\left[D^2\varphi(x+\eps vts)(\eps vt,\eps vt)\right]F(v)\,\dd s\dd t\dd v.
\end{alignat*}
Then, with Cauchy-Schwarz inequality, we can write
\begin{alignat*}{2}
\|I_1\|^2_{L^2_x}&=\eps^{-2\alpha}\!\!\!\int_{\R^d}\left(\int_{|v|\leq C}\!\int_0^{+\infty}\!\!\!\!\int_0^1e^{-t}\left[D^2\varphi(x+\eps vts)(\eps vt,\eps vt)\right]F(v)\,\dd s\dd t\dd v\right)^2\,\dd x\\
&\leq \eps^{-2\alpha}\!\!\!\int_{\R^d}\int_{|v|\leq C}\!\int_0^{+\infty}\!\!\!\!\int_0^1e^{-t}\eps^4|v|^4t^4|D^2\varphi(x+\eps vts)|^2F(v)\,\dd s\dd t\dd v\,\dd x\\
&= \eps^{4-2\alpha}\|D^2\varphi\|^2_{L^2_x}\int_{|v|\leq C}\!\int_0^{+\infty}\!\!\!\!\int_0^1e^{-t}t^4|v|^4F(v)\,\dd s\dd t\dd v\\
&\leq 24C^4 \eps^{4-2\alpha}\|D^2\varphi\|^2_{L^2_x}.
\end{alignat*}
\noindent We are now interested in $I_2$. We first rewrite $I_2$ thanks to the change of variables $w:=\eps vt$
\begin{alignat*}{2}
I_2(x)&=\eps^{-\alpha}\int_0^{+\infty}\int_{|w|\geq C\eps t}e^{-t}\left[\varphi(x+w)-\varphi(x)\right]\frac{\kappa_0|\eps t|^{d+\alpha}}{|w|^{d+\alpha}}\,\frac{dw}{\eps^\dd t^d}\dd t\\
&=\kappa_0\int_0^{+\infty}\int_{|w|\geq C\eps t}e^{-t}|t|^{\alpha}\left[\varphi(x+w)-\varphi(x)\right]\frac{dw}{|w|^{d+\alpha}}\dd t.
\end{alignat*}
Here we recall that the fractional laplacian can be written as
\begin{alignat*}{2}
-(-\Delta)^{\frac{\alpha}{2}}\varphi(x)&=c_{d,\alpha}\mathrm{PV}\int_V\left[\varphi(x+w)-\varphi(x)\right]\frac{dw}{|w|^{d+\alpha}}\\
&=c_{d,\alpha}\int_{|w|\geq 1}\left[\varphi(x+w)-\varphi(x)\right]\frac{dw}{|w|^{d+\alpha}}\\
&\quad +c_{d,\alpha}\int_{|w|\leq 1}\left[\varphi(x+w)-\varphi(x)-w\cdot\nabla_x\varphi(x)\right]\frac{dw}{|w|^{d+\alpha}}\\
&= L_1(x) + L_2(x).
\end{alignat*}
It prompts us to use a similar decomposition of $I_2(x)$; we thus write
\begin{alignat*}{2}
I_2(x)&=\kappa_0\int_0^{1/(C\eps)}e^{-t}|t|^{\alpha}\int_{|w|\geq 1}\left[\varphi(x+w)-\varphi(x)\right]\frac{dw}{|w|^{d+\alpha}}\dd t\\
&+\kappa_0\int_0^{1/(C\eps)}e^{-t}|t|^{\alpha}\int_{C\eps t\leq |w|\leq 1}\left[\varphi(x+w)-\varphi(x)-w\cdot\nabla_x\varphi(x)\right]\frac{dw}{|w|^{d+\alpha}}\dd t\\
&+\kappa_0\int_{1/(C\eps)}^{+\infty}e^{-t}|t|^{\alpha}\int_{|w|\geq C\eps t}\left[\varphi(x+w)-\varphi(x)\right]\frac{dw}{|w|^{d+\alpha}}\dd t\\
&=J_1(x)+J_2(x)+J_3(x).
\end{alignat*}

\noindent We recall the definition $(\ref{defkappa})$ of $\kappa$ :
$$\kappa=\frac{\kappa_0}{c_{d,\alpha}}\int_0^{+\infty}e^{-t}|t|^{\alpha}\,\dd t.$$
Then, with Cauchy-Schwarz inequality,
\begin{alignat*}{2}
\|J_1-\kappa L_1\|^2_{L^2_x}&=\int_{\R^d}\left(\kappa_0\int_{1/(C\eps)}^{+\infty}e^{-t}|t|^{\alpha}\int_{|w|\geq 1}\left[\varphi(x+w)-\varphi(x)\right]\frac{dw}{|w|^{d+\alpha}}\dd t\right)^2\,\dd x\\
&\!\!\!\!\!\!\!\!\!\!\!\!\!\!\!\!\!\!\!\!\!\!\!\!\!\!\!\!\!\!\leq \kappa_0^2\left(\int_{|w|\geq 1}\frac{dw}{|w|^{d+\alpha}}\right)\int_{\R^d}\int_{1/(C\eps)}^{+\infty}e^{-t}|t|^{2\alpha}\int_{|w|\geq 1}\left[\varphi(x+w)-\varphi(x)\right]^2\frac{dw}{|w|^{d+\alpha}}\dd t\,\dd x\\
&\leq 4\kappa_0^2\left(\int_{|w|\geq 1}\frac{dw}{|w|^{d+\alpha}}\right)^2\|\varphi\|^2_{L^2_x}\int_{1/(C\eps)}^{+\infty}e^{-t}|t|^{2\alpha}\,\dd t.
\end{alignat*}
To continue, we decompose $J_2(x)-\kappa L_2(x)$ into two terms $\tau_1(x)+\tau_2(x)$
\begin{multline*}
-\kappa_0\int_0^{1/(C\eps)}e^{-t}|t|^{\alpha}\int_{0\leq|w|\leq C\eps t}\int_0^1D^2\varphi(x+ws)(w,w)\dd s\frac{dw}{|w|^{d+\alpha}}\dd t\\
-\kappa_0\int_{1/(C\eps)}^{+\infty}e^{-t}|t|^{\alpha}\int_{|w|\leq 1}\int_0^1D^2\varphi(x+ws)(w,w)\dd s\frac{dw}{|w|^{d+\alpha}}\dd t.
\end{multline*}
We work on $\|\tau_1\|^2_{L^2_x}$, using Cauchy-Schwarz inequality, and the change of variables $v=w/(\eps t)$:
\begin{alignat*}{2}
\|\tau_1\|^2_{L^2_x}&=\int_{\R^d}\left(\kappa_0\int_0^{1/(C\eps)}e^{-t}|t|^{\alpha}\int_{0\leq|w|\leq C\eps t}\int_0^1D^2\varphi(x+ws)(w,w)\dd s\frac{dw}{|w|^{d+\alpha}}\dd t\right)^2\,\dd x\\
&\leq \int_{\R^d}\left(\kappa_0\int_0^{1/(C\eps)}e^{-t}|t|^{\alpha}\int_{0\leq|w|\leq C\eps t}\int_0^1|D^2\varphi(x+ws)|\dd s\frac{dw}{|w|^{d+\alpha-2}}\dd t\right)^2\,\dd x\\
&\!\!\!\!\!\!\!\!\!\!\!\!\!\!\!\!\!\!\leq\kappa_0^2\int_{|w|\leq 1}\frac{dw}{|w|^{d+\alpha-2}} \int_{\R^d}\int_0^{1/(C\eps)}e^{-t}|t|^{2\alpha}\int_{0\leq|w|\leq C\eps t}\int_0^1|D^2\varphi(x+ws)|^2\dd s\frac{dw}{|w|^{d+\alpha-2}}\dd t\,\dd x\\
&\!\!\!\!\!\!\!\!\!\!\!\!\!\!\!\!\!\!\leq\kappa_0^2\int_{|w|\leq 1}\frac{dw}{|w|^{d+\alpha-2}} \|D^2\varphi\|^2_{L^2_x}\int_0^{+\infty}e^{-t}|t|^{2\alpha}\int_{0\leq|w|\leq C\eps t}\frac{dw}{|w|^{d+\alpha-2}}\dd t\\
&=\kappa_0^2\int_{|w|\leq 1}\frac{dw}{|w|^{d+\alpha-2}}\int_0^{+\infty}e^{-t}|t|^{\alpha+2}\dd t\int_{|v|\leq C}\frac{\dd v}{|v|^{d+\alpha-2}}\eps^{2-\alpha} \|D^2\varphi\|^2_{L^2_x}.
\end{alignat*}
With the same kind of computations, we are led to
$$\|\tau_2\|^2_{L^2_x}\leq\kappa_0^2\left(\int_{|w|\leq 1}\frac{dw}{|w|^{d+\alpha-2}}\right)^2\|D^2\varphi\|^2_{L^2_x}\int_{1/(C\eps)}^{+\infty}e^{-t}|t|^{2\alpha}\,\dd t,$$
and
$$\|J_3\|^2_{L^2_x}\leq 4\kappa_0^2\left(\int_{|w|\geq 1}\frac{dw}{|w|^{d+\alpha}}\right)^2\|\varphi\|^2_{L^2_x}\int_{1/(C\eps)}^{+\infty}e^{-t}|t|^{2\alpha}\,\dd t.$$
Finally, about the case of $I_3$, thanks to $(\ref{asymp})$, we can do the same work as for $I_2$; then we just have to put together all the bounds obtained to get the result. This concludes the proof. $\Box$
\end{proof}

\begin{proof}[Proof of Lemma \ref{lemmeLaplacien}]
First, we write
\begin{alignat*}{2}
\eps^{-\alpha}(Lf+\eps Af,\chi^{\eps}F)&=\eps^{-\alpha}\int_{\R^d}\int_V\rho F\chi^{\eps}-f\chi^{\eps} -\eps v\cdot\nabla_xf\chi^{\eps}\,\dd v\dd x \\
&=\eps^{-\alpha}\int_{\R^d}\int_V\rho F\chi^{\eps}-f(\chi^{\eps}-\eps v\cdot\nabla_x\chi^{\eps})\,\dd v\dd x \\
&=\eps^{-\alpha}\int_{\R^d}\int_V\rho F\chi^{\eps}-f\varphi\,\dd v\dd x =\int_{\R^d}\rho\int_V\eps^{-\alpha}\left[\chi^{\eps}-\varphi\right]F\,\dd v\dd x,
\end{alignat*}
where we used an integration by part and $(\ref{aux})$. Furthermore, we have 
\begin{align*}
(-\kappa(-\Delta)^{\frac{\alpha}{2}}f,\varphi F)&=(f,-\kappa(-\Delta)^{\frac{\alpha}{2}}\varphi F)\\
&=-\kappa\int_{\R^d}\int_V f(-\Delta)^{\frac{\alpha}{2}}\varphi\,\dd v\dd x \\
&=-\kappa\int_{\R^d}\rho(-\Delta)^{\frac{\alpha}{2}}\varphi\,\dd v\dd x.
\end{align*}
As a consequence, with Cauchy-Schwarz inequality, we get
\begin{align*}
\eps^{-\alpha}(\eps Af+Lf,\chi^{\eps}F)+(\kappa(-\Delta)^{\frac{\alpha}{2}}f,\varphi F)
&=\int_{\R^d}\rho\left[\int_V\eps^{-\alpha}\left[\chi^{\eps}-\varphi\right]F\dd v+\kappa(-\Delta)^{\frac{\alpha}{2}}\varphi\right]\dd x\\
&\!\!\!\!\!\!\!\!\!\!\!\!\!\!\!\!\!\!\!\!\leq \|\rho\|_{L^2_x}\left\|\int_V\eps^{-\alpha}\left[\chi^{\eps}-\varphi\right]F\dd v+\kappa(-\Delta)^{\frac{\alpha}{2}}\varphi\right\|_{L^2_x} \\
&\!\!\!\!\!\!\!\!\!\!\!\!\!\!\!\!\!\!\!\!\leq \|f\|\left\|\int_V\eps^{-\alpha}\left[\chi^{\eps}-\varphi\right]F\dd v+\kappa(-\Delta)^{\frac{\alpha}{2}}\varphi\right\|_{L^2_x},
\end{align*}
and an application of Lemma $\ref{lemmelaplacienL2}$ then concludes the proof. $\Box$
\end{proof}

\bibliography{biblio}
\end{document}